\documentclass[12pt,reqno]{amsart}

\input{xy}
\xyoption{all}
\usepackage{graphicx}
\usepackage{color}
\usepackage{verbatim}
\usepackage{amsthm}
\usepackage{amsmath}
\usepackage{amssymb}
\usepackage{amsfonts}
\usepackage{tikz}
\usepackage{tikz-cd}
\usepackage{float}
\usepackage[bookmarks,colorlinks,breaklinks]{hyperref}  
\hypersetup{linkcolor=blue,citecolor=blue,filecolor=blue,urlcolor=blue}

\newtheorem{theorem}{Theorem} [section]
\newtheorem{prop}[theorem]{Proposition}
\newtheorem{lemma}[theorem]{Lemma}

\newtheorem{conjecture}[theorem]{Conjecture}
\newtheorem{question}[theorem]{Question}

\theoremstyle{definition}

\newtheorem{remark}[theorem]{Remark}

\numberwithin{equation}{section}
\numberwithin{figure}{section}

\newcommand\A{{\mathbb A}}
\newcommand\C{{\mathbb C}}
\newcommand\Chat { {\hat{\C}} } 

\renewcommand\P{{\mathbb P}}

\newcommand\Q{{\mathbb Q}}

\newcommand\D{{\mathbb D}}

\newcommand\Gal{\operatorname{Gal}}

\newcommand\Qbar{\overline{\mathbb{Q}}}

\newcommand\Kbar{\overline{K}}
\newcommand\M {M} 
\newcommand\cM{\mathcal{M}}

\newcommand\Preper {\mathrm{Preper}}

\newcommand\<{\langle} 
\renewcommand\>{\rangle} 
\newcommand\eps{\varepsilon}

\newcommand\dist{\mathrm{dist}}
\newcommand\Diag{\mathrm{Diag}}

\newcommand\Mbad{M_{\mathrm{close}}}
\newcommand\Mgood{M_{\mathrm{help}}}
\newcommand\Mbounded{M_{\mathrm{bounded}}}

\begin{document}

\author{Laura De~Marco}
\address{
Laura DeMarco\\
Department of Mathematics\\
Harvard University\\
1 Oxford Street\\
Cambridge, MA  02138 \\
USA
}
\email{demarco@math.harvard.edu}

\author{Holly Krieger}
\address{
Holly Krieger\\
Department of Pure Mathematics and Mathematical Statistics\\
 University of Cambridge \\
Cambridge CB3 0WB\\
UK
}
\email{hkrieger@dpmms.cam.ac.uk}

\author{Hexi Ye}
\address{
Hexi Ye\\
Department of Mathematics\\
Zhejiang University\\
Hangzhou, 310027\\
China}
\email{yehexi@gmail.com}

\date{\today}

\begin{abstract}
Let $f_c(z) = z^2+c$ for $c \in \C$.  We show there exists a uniform upper bound on the number of points in $\P^1(\C)$ that can be preperiodic for both $f_{c_1}$ and $f_{c_2}$, for any pair $c_1\not= c_2$ in $\C$.  The proof combines arithmetic ingredients with complex-analytic: we estimate an adelic energy pairing when the parameters lie in $\Qbar$, building on the quantitative arithmetic equidistribution theorem of Favre and Rivera-Letelier, and we use distortion theorems in complex analysis to control the size of the intersection of distinct Julia sets.  The proofs are effective, and we provide explicit constants for each of the results. 
\end{abstract}

\title{Common preperiodic points for quadratic polynomials}
\maketitle
\thispagestyle{empty}

\section{Introduction}

Consider the family of quadratic polynomials 
	$$f_c(z) = z^2 + c$$
for $c \in \C$, each viewed as a dynamical system $f_c :\Chat \to \Chat$ on the Riemann sphere.  Recall that a point $z \in \Chat$ is said to be \emph{preperiodic} if its forward orbit under $f_c$ is finite.  It is well known that the set of all preperiodic points for $f_c$ will determine $c$.  Indeed, we have 
\begin{equation} \label{known equivalence}
	\Preper(f_{c_1}) = \Preper(f_{c_2}) \iff J(f_{c_1}) = J(f_{c_2}) \iff c_1 = c_2
\end{equation}
in this family, where $J(f_c)$ is the Julia set and $\Preper(f_c)$ the set of preperiodic points \cite{Baker:Eremenko}; see \S\ref{unique} for more information.

For any $c_1\not= c_2$ in $\C$, the intersection of $\Preper(f_{c_1})$ and $\Preper(f_{c_2})$ is finite \cite[Corollary 1.3]{BD:preperiodic} \cite[Theorem 1.3]{Yuan:Zhang:II}, even though their respective accumulation sets, the Julia sets of $f_{c_1}$ and $f_{c_2}$, can have complicated, infinite intersection.  We investigate the question of how many preperiodic points are required to uniquely determine the polynomial, without the information of the period or length of an orbit.  We prove:  

\begin{theorem} \label{preperbound} 
There exists a uniform constant $B$ so that
$$|\Preper(f_{c_1}) \cap \Preper(f_{c_2})| \leq B$$
for every $c_1 \not= c_2$ in $\C$.  
\end{theorem}

\begin{remark}  \label{explicit B}
Our proof leads to an explicit value for $B$.  Without making an effort to optimize our constants, we show that we can take $B = 10^{103}$.  This bound is probably far from optimal.  The largest intersection we know was found by Trevor Hyde:  the set $\Preper(f_{-21/16}) \cap \Preper(f_{-29/16})$ consists of at least 27 points in $\Chat$.  These two polynomials also appear in \cite{Poonen:conjecture}.
\end{remark}

\begin{remark} 
There is no uniform bound on the periods or orbit lengths of the elements of $\Preper(f_{c_1}) \cap \Preper(f_{c_2})$ as $c_1$ and $c_2$ vary.  For example, taking $c_1$ and $c_2$ to be distinct centers of hyperbolic components within the Mandelbrot set, we will have $0 \in \Preper(f_{c_1}) \cap \Preper(f_{c_2})$ with periods as large as desired.   
\end{remark}

\subsection{Motivation and background}
For any pair of rational functions $f, g: \Chat \to \Chat$ of degree at least 2, it is known that a dichotomy holds:  either the intersection $\Preper(f) \cap \Preper(g)$ is finite or $\Preper(f) = \Preper(g)$ \cite{BD:preperiodic, Yuan:Zhang:II}.  We suspect a much stronger result may hold, and we propose the following conjecture:

\begin{conjecture}  \label{conjecture bound}
For each degree $d\geq 2$, there exists a constant $B = B(d)$ so that either
$$|\Preper(f) \cap \Preper(g)| \leq B$$
or 
	$$\Preper(f) = \Preper(g)$$
for any pair of rational functions $f$ and $g$ in $\C(z)$ of degree $d$.  
\end{conjecture}

Conjecture \ref{conjecture bound} would imply that a configuration of $B+1$ points on the Riemann sphere, if preperiodic for some map of degree $d\geq 2$, will almost uniquely determine the map among all maps of the same degree.  It is known that, except for maps conjugate to $z^{\pm d}$, the equality $\Preper(f) = \Preper(g)$ is equivalent to the statement that the measures of maximal entropy for $f$ and $g$ coincide; one implication is proved in \cite{Levin:Przytycki} (assuming $f$ and $g$ are non-exceptional) and the other in \cite[Theorem 1.5]{Yuan:Zhang:II}.    A complete classification of all rational maps having the same measure of maximal entropy is still open, however, unless the maps are polynomial  \cite{Baker:Eremenko, Beardon:symmetries}; see also \cite{Levin:Przytycki, Ye:symmetries, Pakovich:maximal} for results about rational maps with the same maximal measure.

As discussed in \cite{DKY:UMM}, Conjecture \ref{conjecture bound} is analogous to a question posed by Mazur  \cite{Mazur:curves}, proposing the existence of uniform bound -- depending only on the genus $g$ -- on the number of torsion points on a compact Riemann surface of genus $g>1$ inside its Jacobian.  In fact, the special case of Conjecture \ref{conjecture bound} for the 1-parameter family of Latt\`es maps
\begin{equation} \label{Lattes}
	f_t(z) = \frac{(z^2-t)^2}{4z(z-1)(z-t)}
\end{equation}
in degree 4, for $t \in \C\setminus\{0,1\}$, was proved in \cite{DKY:UMM}; it implied a positive answer to Mazur's question for a certain 2-parameter family of genus 2 Riemann surfaces.  (The uniform bound for all curves of a fixed genus was recently obtained by K\"uhne \cite{Kuhne:UMM}.)

\begin{remark} The bound $B$ in Conjecture \ref{conjecture bound}, if it exists, must depend on the degree $d$.  It is easy to find examples with growing degrees with growing numbers of common preperiodic points.  For example, the sequences of polynomials
	$$f_n(z) = z^2(z-1)\cdots(z-n) \quad\mbox{ and } \quad g_n(z) = z(z-1)\cdots(z-n)(z-(n+1))$$
have degree $n+2$ with at least $n+1$ common preperiodic points, for all $n\geq 1$.  Their sets of preperiodic points cannot be equal because their Julia sets are not the same:  we have $0 \in J(g_n)$ for all $n\geq 1$, because the fixed point at $0$ is repelling for $g_n$, but $0 \not\in J(f_n)$ for all $n$, because the fixed point at $0$ is attracting for $f_n$.
\end{remark}

\subsection{Further results and proof strategy}
The proof of Theorem \ref{preperbound} employs a combination of arithmetic and analytic techniques, and we first prove a version of Theorem \ref{preperbound} when the parameters $c_1$ and $c_2$ are algebraic numbers.   The basic algebraic observation is that the set of preperiodic points of $f_c$ is invariant under the action of the Galois group $\Gal(\Kbar/K)$, for any number field $K$ containing $c$.  Finiteness of the intersection $\Preper(f_{c_1}) \cap \Preper(f_{c_2})$, when $c_1$ and $c_2$ are algebraic, is an immediate consequence of arithmetic equidistribution:  large Galois orbits in the set $\Preper(f_c)$ are uniformly distributed with respect to the measure of maximal entropy $\mu_c$ \cite{Baker:Rumely:equidistribution, FRL:equidistribution, ChambertLoir:equidistribution}, while $\mu_{c_1} = \mu_{c_2}$ if and only if $c_1 = c_2$ \cite{Baker:Eremenko}.  We provide a few simple examples in Section \ref{examples} to illustrate these ideas. 

The uniform bound in Theorem \ref{preperbound} comes from controlling the rate of equidistribution, not just over $\C$ but at all places of the number field $K$ simultaneously.  To do so, we make use of an adelic energy pairing between the polynomials $f_{c_1}$ and $f_{c_2}$.  This is a sum of integrals, one for each of the primes associated to a number field $K$ containing both $c_1$ and $c_2$, which we describe now.  For any $c$ in $K$ and any place $v$ of $K$, we let 
	$$\lambda_{c,v}(z) = \lim_{n\to\infty} \frac{1}{2^n} \log \max\{ |f_c^n(z)|_v, 1\}$$
denote the $v$-adic escape-rate function of $f$, with $z$ in the field of $v$-adic numbers $\C_v$.  This is the usual escape-rate function on $\C$ for $v | \infty$, coinciding with the Green's function for the complement of the filled Julia set, with logarithmic pole at $\infty$.  At every place $v$, the function $\lambda_{c,v}$ extends continuously and subharmonically to the Berkovich affine line $\A^{1,an}_v$, and its Laplacian is the canonical $v$-adic measure $\mu_{c,v}$ for $f_c$ \cite{BRbook, FRL:equidistribution}.  For archimedean places $v$, we recover the Brolin-Lyubich measure \cite{Brolin, Lyubich:entropy}.  The energy pairing is defined to be
\begin{equation}  \label{energy pairing}
\< f_{c_1}, f_{c_2}\>  := \sum_{v \in M_K}  \frac{[K_v:\Q_v]}{[K:\Q]} \int_{\A^{1, an}_v} \lambda_{c_1,v} \, d\mu_{c_2, v}.
\end{equation}
The pairing is symmetric, and each term in the sum is non-negative, vanishing if and only if $\mu_{c_1, v} = \mu_{c_2,v}$ \cite{PST:pairing}.  The integral thus provides a notion of distance between the two measures.  In particular, we have 
	$$\<f_{c_1}, f_{c_2} \> \geq 0 \mbox{ with equality if and only if } c_1 = c_2.$$

We prove:  

\begin{theorem}\label{uniformbound}
There is a constant $\delta>0$, such that 
   $$\<f_{c_1},f_{c_2}\> \geq \delta$$
for all $c_1\neq c_2\in \Qbar$. 
\end{theorem}

\noindent 
In other words, two Julia sets cannot be too similar at {\em all} places of a given number field.  See \S\ref{effectiveness} for comments on the magnitude of $\delta$ and the constants in the following theorem.

\begin{theorem}\label{pairingbound}
There are constants $\alpha_1, \alpha_2, C_1, C_2 >0 $ so that 
   $$   \alpha_1 \, h(c_1, c_2)-C_1 \leq \<f_{c_1},f_{c_2}\>\leq \alpha_2 \, h(c_1,c_2)+C_2,$$
for all $c_1\neq c_2$ in $\Qbar$, where $h$ is the logarithmic Weil height on $\A^2(\Qbar)$.
\end{theorem}

The upper bound on $\<f_{c_1},f_{c_2}\>$ in Theorem \ref{pairingbound} is straightforward to prove, and it is also fairly easy to obtain a weaker lower bound in terms of $h(c_1-c_2)$, the Weil height of the difference, in place of the height $h(c_1, c_2)$; see Theorem \ref{weak pairingbound}.  The lower bound of Theorem \ref{pairingbound} is more delicate: see Section \ref{stronglowerbound}.

Finally, we relate the energy pairing to the number of common preperiodic points via a quantified version of the arithmetic equistribution theorems, building upon ideas of Favre, Rivera-Letelier, and Fili \cite{FRL:equidistribution} \cite{Fili:energy}:

\begin{theorem} \label{boundingN} 
For all $0 < \eps < 1$, there exists a constant $C(\eps) > 0$ so that
$$\<f_{c_1},f_{c_2}\> \leq \left( \eps + \frac{C(\eps)}{N(c_1, c_2)-1} \right) (h(c_1, c_2) + 1)$$
for all $c_1 \not= c_2$ in $\Qbar$ with 
	$$N(c_1, c_2) := |\Preper(f_{c_1}) \cap \Preper(f_{c_2})| > 1.$$
\end{theorem}

\noindent
Note that $N(c_1, c_2) \geq 1$ for every $c_1$ and $c_2$, because $\infty$ is a fixed point for every $f_c$.
Using standard distortion estimates in complex analysis to control the archimedean contributions to the pairing, our proof shows that we can take 
	$$C(\eps) \asymp \log(1/\eps)$$
in Theorem \ref{boundingN}.

Theorems \ref{uniformbound},  \ref{pairingbound}, and \ref{boundingN} combine to give a uniform upper bound on the number $N(c_1, c_2)$ for all $c_1 \not= c_2$ in $\Qbar$, thus proving Theorem \ref{preperbound} for $c_1$ and $c_2$ and $\Qbar$.  Once a uniform bound is obtained over $\Qbar$, it is a standard specialization argument to show the same bound holds over $\C$, as we explain in \S\ref{proof over C}, which completes the proof of Theorem \ref{preperbound}.

\subsection{Comparison with \cite{DKY:UMM}}
This general strategy of proof was introduced in our earlier work \cite{DKY:UMM}, and the reader will recognize the similarities between the statements of Theorems \ref{uniformbound}, \ref{pairingbound}, and \ref{boundingN} here and Theorems 1.6, 1.7, and 1.8 of \cite{DKY:UMM}.  However, there are significant technical differences between the proofs, and, perhaps more importantly, the proof strategy in this article  is effective; we explain how to obtain a value for $B$ in Theorem \ref{preperbound}.  Regarding the technical aspects of the proofs, in the setting of \cite{DKY:UMM}, the energy integrals at non-archimedean places could be computed explicitly; here, we can only obtain estimates.  For the computations at the archimedean places, the local heights (escape rates of the polynomials) are not smooth for the polynomials considered here, and the shrinking H\"older exponents (as $c\to \infty$) leads to the loss of uniformity in rates of convergence in the equidistribution theorems.  We make use of classical complex dynamical methods in this article such as the Koebe 1/4-theorem; by contrast, in \cite{DKY:UMM}, we obtained the archimedean estimates through the use of degeneration theory and comparison to a limiting non-archimedean dynamical system associated to a function field, as carried out in \cite{Favre:degenerations} and \cite{DF:degenerations, DF:degenerations2}.  The degeneration theory might be used here as well, but at the expense of losing the effective bounds.

As in the setting of \cite{DKY:UMM}, our proofs are as much about the associated canonical height functions $\hat{h}_c$ on $\P^1(\Qbar)$, for $f_c$ with $c\in \Qbar$, as about preperiodic points; the bound of Theorem \ref{preperbound} comes from the fact that $\hat{h}_c(x) = 0$ if and only if $x$ is preperiodic for $f_c$ \cite[Corollary 1.1.1]{Call:Silverman}.  Though we do not provide all the details, it is possible to prove a stronger statement about points of small height:  there exist uniform constants $B$ and $b>0$ so that  $\left|\{x \in \P^1(\Qbar):  \; \hat{h}_{c_1}(x) + \hat{h}_{c_2}(x) \leq b\}\right| \; \leq\;  B$ for all $c_1 \not= c_2$ in $\Qbar$.  A version of this statement is proved for the Latt\`es family \eqref{Lattes} in \cite[Theorems 1.8 and 8.1]{DKY:UMM}.

\subsection{Effectiveness}  \label{effectiveness}
We illustrate the effectiveness of our method by providing explicit constants for each of the theorems stated above.   The proof of Theorem \ref{pairingbound} shows that we can take $\alpha_1 = 1/192$,  $C_1 =  3/17$, $\alpha_2 = 1/2$ and $C_2 = 7/3$.  The proof of Theorem \ref{boundingN} provides $C(\eps) = 40 \log(25/\eps)$.  The first proof of Theorem \ref{uniformbound} that we present in  \S\ref{weaklowerbound} is not sufficient to provide an explicit value for the $\delta$ of Theorem \ref{uniformbound}, but further control on the classical (archimedean) energy pairing leads to $\delta = 10^{-96}$ in \S\ref{explicit delta}.  This exceptionally small $\delta$ gives rise to the enormous bound $B = 10^{103}$ in Theorem \ref{preperbound} that was stated in Remark \ref{explicit B}.  Few examples of $\<f_{c_1}, f_{c_2}\>$ have been computed explicitly; it was recently shown that $\<f_0, f_{-1}\> \approx 0.168$ \cite[\S8]{Andrews:Petsche}, and it would be interesting to see other values.

\subsection{Height pairings}  
The energy pairing $\<f_{c_1}, f_{c_2}\>$ that we work with is a special case of a more general construction, the Arakelov-Zhang pairing, an arithmetic intersection number between adelically metrized line bundles; see \cite{Zhang:adelic}, \cite{PST:pairing}, and \cite{ChambertLoir:survey}.  In this case, each $f_c$ with $c$ in a number field $K$ gives rise to a family of metrics on $O_{\P^1}(1)$, one for each place $v$ of $K$, with non-negative curvature distribution equal to the canonical measure $\mu_{c,v}$ on the Berkovich projective line ${\bf P}^{1,an}_v$.  Each such adelic metric then gives rise to a height function $\hat{h}_{c}$ on $\P^1(\Qbar)$, recovering the dynamical canonical height for $f_c$ of Call and Silverman \cite{Call:Silverman}.

There are other natural height pairings that one could consider for $c_1, c_2\in \Qbar$.  For example, Kawaguchi and Silverman study
	$$[f,g]_{KS} := \sup_{x \in \P^1(\Qbar)} \left|\hat{h}_f(x) - \hat{h}_g(x) \right|$$
for any pair of maps $f,g: \P^1\to\P^1$ defined over $\Qbar$  \cite{Kawaguchi:Silverman:pairing}.  As a consequence of arithmetic equidistribution, we see that 
\begin{equation} \label{upper KS}
	\<f, g\> \leq [f,g]_{KS}.
\end{equation}
Indeed, along any infinite (non-repeating) sequence $x_n \in \P^1(\Qbar)$ for which $\hat{h}_f(x_n) \to 0$, we have by equidistribution that $\hat{h}_g(x_n) \to \< f, g\>$ \cite[Theorem 1]{PST:pairing}.  Such sequences always exist (the preperiodic points of $f$ will have height 0), so we obtain \eqref{upper KS}.  We do not know if a similar inequality always holds in the reverse direction for any pair of maps.  However, as a corollary of Theorem \ref{pairingbound}, we have

\begin{theorem} \label{KS comparison}
There exist constants $\alpha, C>0$ so that 
	$$\alpha [f_{c_1},f_{c_2}]_{KS}	- C \; \leq \; \<f_{c_1}, f_{c_2}\> \; \leq \; [f_{c_1}, f_{c_2}]_{KS}$$
for all $c_1, c_2 \in \Qbar$.
\end{theorem}

\begin{proof}
From \cite[Theorem 1]{Kawaguchi:Silverman:pairing}, we have $[f_{c_1},f_{c_2}]_{KS} \leq \kappa_1 (h(c_1) + h(c_2)) + \kappa_2$ for constants $\kappa_1, \kappa_2$ depending only on the degrees of the maps, and the definition of the Weil height shows that $h(c_1) + h(c_2) \leq 2 h(c_1, c_2)$.  The lower bound of the theorem then follows immediately from the lower bound in Theorem \ref{pairingbound}.
\end{proof}

A version of Theorem \ref{KS comparison} also holds for the Latt\`es family $f_t(z) = (z^2-t)^2/(4z(z-1)(z-t))$, with $t_1, t_2 \in \Qbar\setminus\{0,1\}$, as a consequence of \cite[Theorem 1.5]{DKY:UMM}.

\begin{question}
Do we have 
	$$\<f,g\> \asymp [f,g]_{KS}$$
for all maps $f,g: \P^1\to \P^1$, defined over $\Qbar$, with constants depending only on the degrees of $f$ and $g$?
\end{question}

\subsection{Outline}
Section \ref{examples} illustrates some basic examples towards understanding the content of Theorem \ref{preperbound}.  Local estimates on the pairing are carried out in Sections \ref{archestimates} -- \ref{2p}.  In Section \ref{pairingboundsection}, we prove Theorem \ref{uniformbound}, and in Section \ref{stronglowerbound} we prove Theorem \ref{pairingbound}.  Theorem \ref{boundingN} is proved via quantitative equidistribution theory in Section \ref{QE}, and Section \ref{mainproof} establishes our main result, Theorem \ref{preperbound}.  Finally, in Section \ref{effective} we make all bounds effective.

\subsection{Acknowledgements} 
We thank the American Institute of Mathematics, where the initial work for this paper took place as part of an AIM SQuaRE.  We thank Khoa Nguyen, especially for ideas related to the proof of Theorem \ref{uniformbound}, and we thank Joe Silverman for helpful conversations.  We also thank Hang Fu and the anonymous referees for catching errors in an earlier version and providing many helpful suggestions.  During the preparation of this paper, L.~DeMarco was supported by the National Science Foundation (DMS-1600718, DMS-1856103, DMS-2050037), H.~Krieger was partially supported by Isaac Newton Trust (RG74916), and H.~Ye was partially supported by ZJNSF (LR18A010001) and NSFC (11701508).

\bigskip
\section{Basic examples}
\label{examples}

Let $f_c(z) = z^2 + c$, for $c \in \C$.  Note that 
	$$| \Preper(f_{c_1}) \cap \Preper(f_{c_2})| \geq \; 1$$
for every pair, because the sets always contain the point at $\infty$.  Here we provide a few simple examples, illustrating some of the ideas that appear in our proof of Theorem \ref{preperbound}.  Recall that the \emph{filled Julia set} of $f_c$ is 
	$$K(f_c) = \left\{z \in \C ~:~  \sup_{n\geq 1} |f^n_c(z)|  < \infty\right\},$$
and the Julia set satisfies $J(f_c) =  \partial K(f_c)$.  

\subsection{Disjoint filled Julia sets}
When two quadratic polynomials $f_{c_1}$ and $f_{c_2}$ have disjoint filled Julia sets, they have no common preperiodic points other than $\infty$.  Sometimes the filled Julia sets have nontrivial intersection in $\C$, but -- for algebraic parameters -- it may be that the $v$-adic filled Julia sets are disjoint at some place $v$.  In that case, again, there can be no common preperiodic points other than $\infty$.  Examples are shown in Figures \ref{over C} and \ref{at 5}.  As we shall explain in the following sections, the filled Julia set of $f_c$ at any non-archimedean place $v$ (defined as the set of points with bounded orbit in the Berkovich affine line ${\bf A}^1_v$, over the field $\C_v$) with $|c|_v >1$ is a subset of $\{z \in {\bf A}^1_v: |z|_v = |c|_v^{1/2}\}$, while it is the closed unit disk whenever $|c|_v \leq 1$.  

\begin{figure} [h]
\includegraphics[width=1.95in]{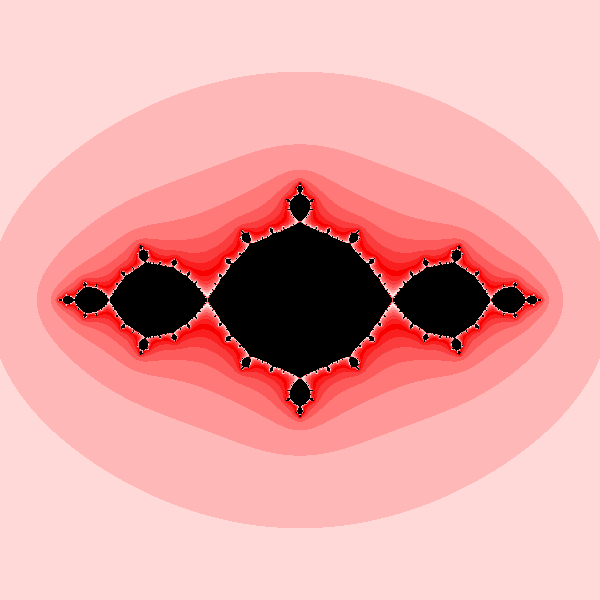}
\includegraphics[width=1.95in]{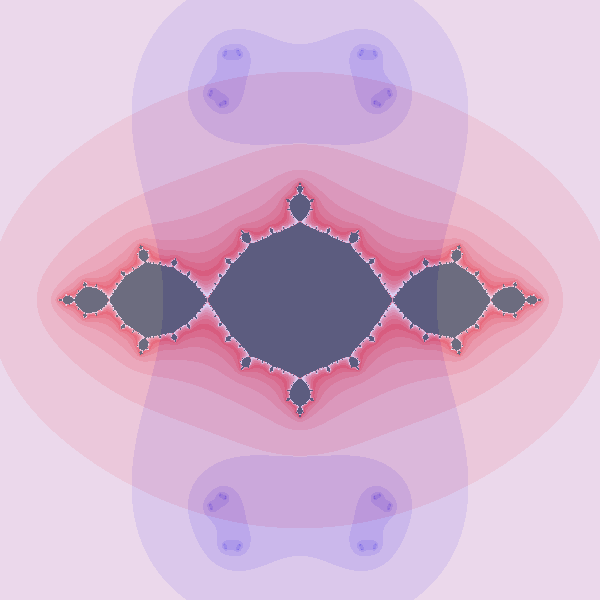}
\includegraphics[width=1.95in]{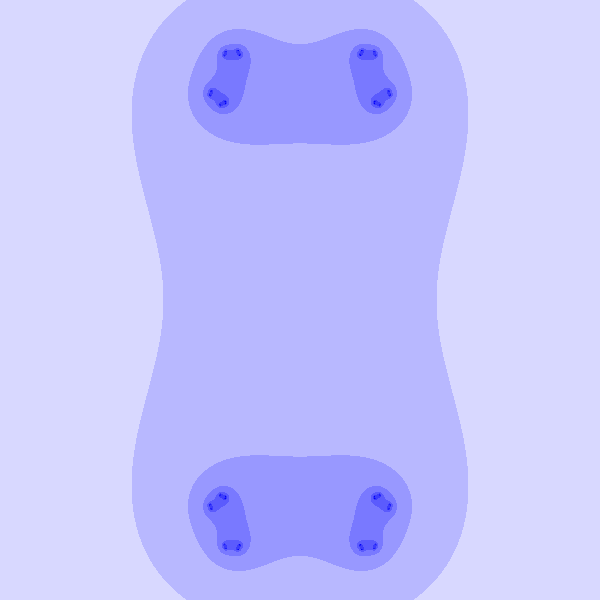}
\caption{ \small The filled Julia sets of $f(z) = z^2 - 1$ (left) and $g(z) = z^2 +2$ (right) are disjoint; they have no common preperiodic points except for $\infty$.}
\label{over C}
\end{figure}

\begin{figure} [h]
\includegraphics[width=1.95in]{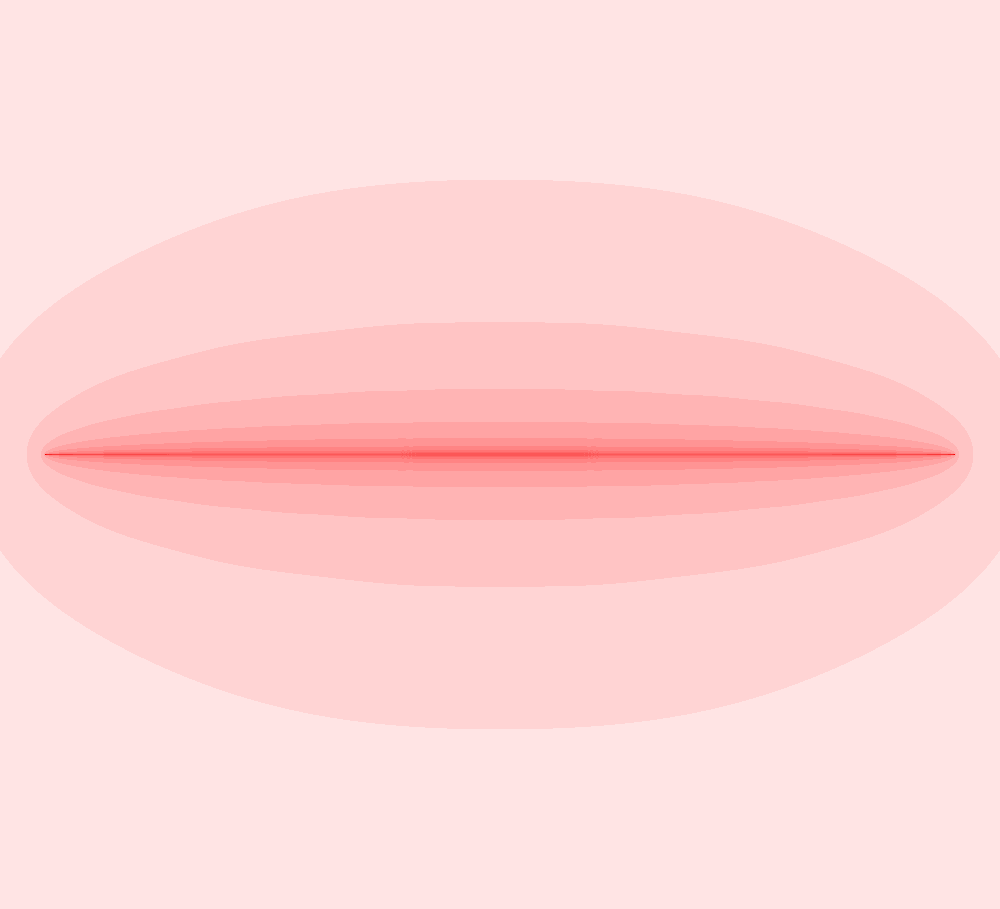}
\includegraphics[width=1.95in]{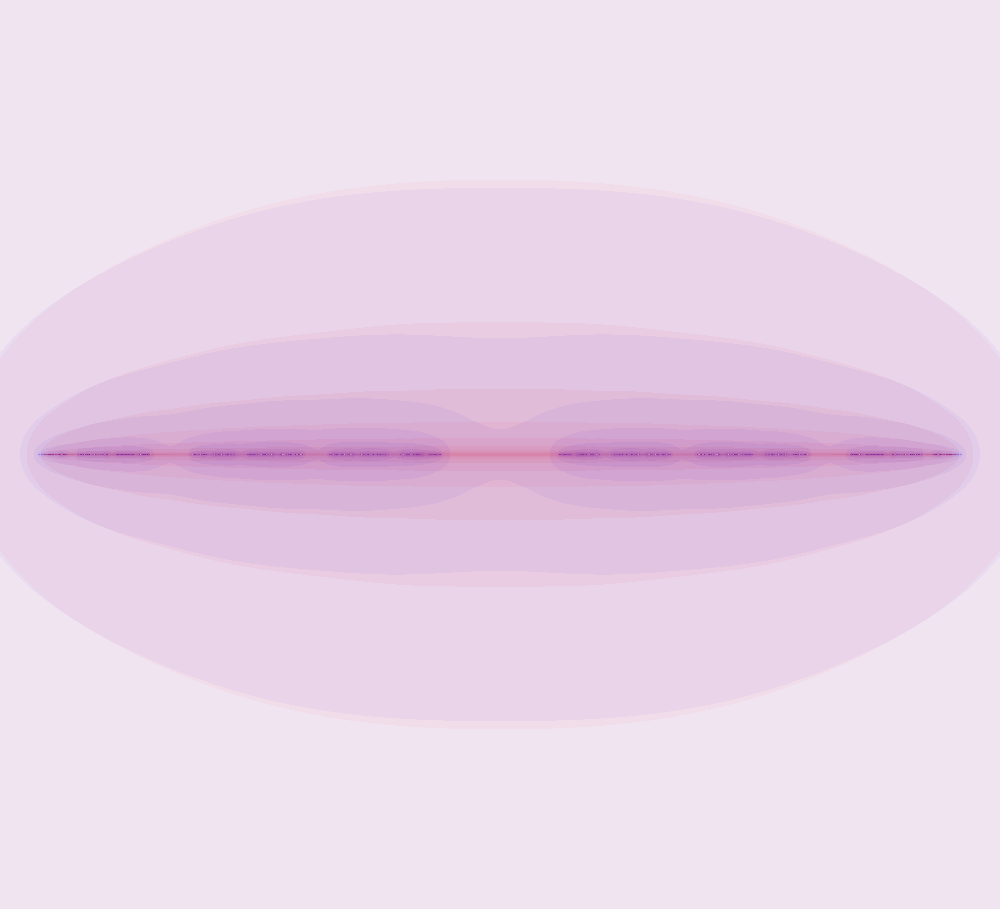}
\includegraphics[width=1.95in]{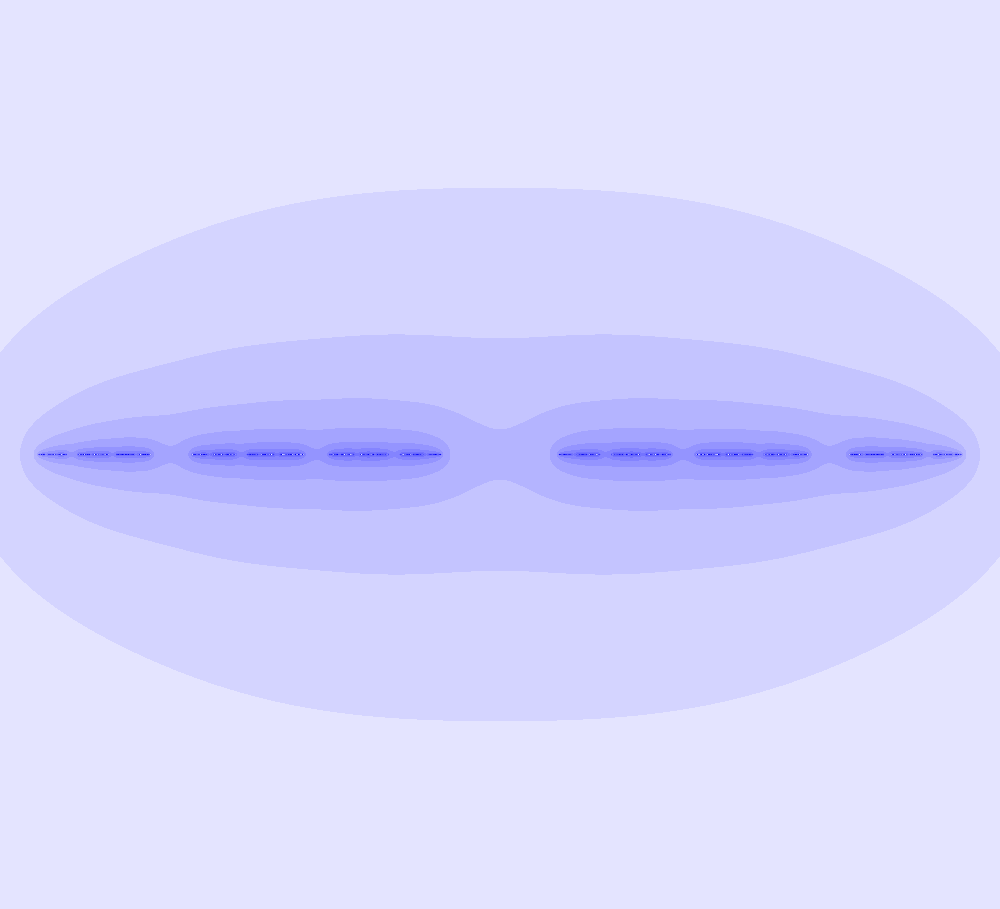}\\
\caption{ \small The filled Julia sets of $f(z) = z^2 - 2$ (left) and $g(z) = z^2 -2.1$ (right) have significant overlap in $\C$, but there are no common preperiodic points except for $\infty$, because the filled Julia sets are disjoint at the primes 2 and 5.}
\label{at 5}
\end{figure}

\subsection{Galois orbits}
Let $f(z) = z^2$ and $g(z) = z^2 - 1$; see Figure \ref{circle}.  Here we show that
		$$ \Preper(f) \cap \Preper(g) \; = \; \{0,1, -1, \infty\}.$$
We know that the preperiodic points of $f$ are the roots of unity, together with 0 and $\infty$.  The preperiodic points of any $f_c$ are roots of the polynomial equations given by $f_c^n(z) = f^m_c(z)$ for any $n > m \geq 0$; so the set of preperiodic points is invariant under the action of $\Gal(\Kbar/K)$, whenever $c$ lies in a number field $K$.  In this case, we can take $K = \Q$.  So we need to show that for all $n\geq 3$, at least one of the primitive $n$-th roots of unity will have infinite forward orbit under the action of $g$.

\begin{figure} [h]
\includegraphics[width=3.49in]{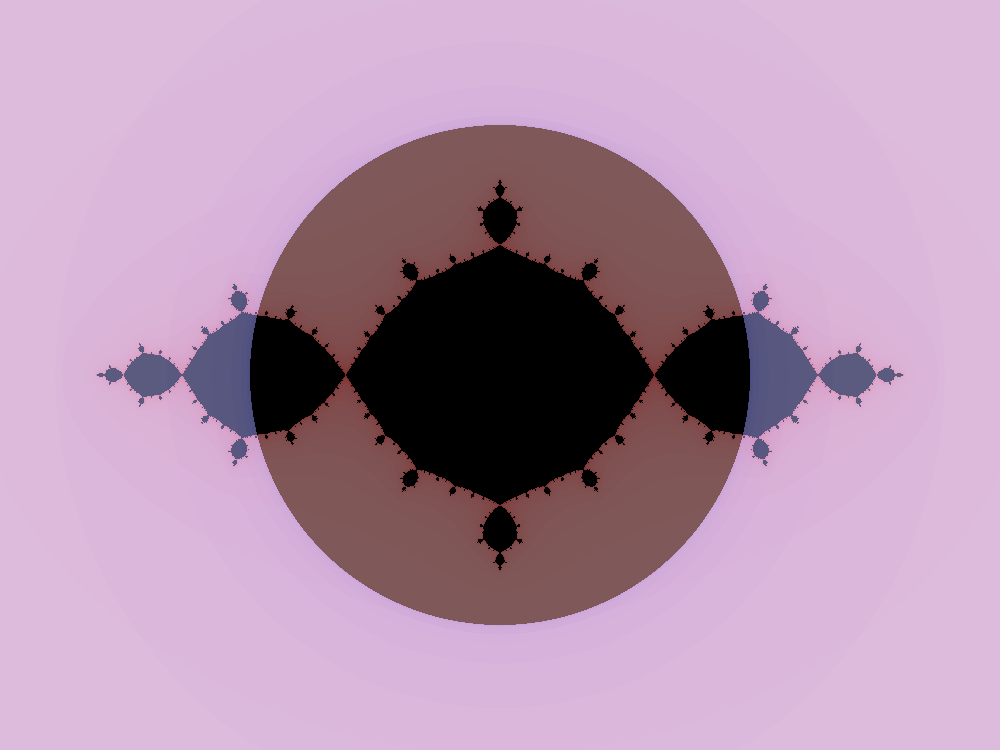}
\includegraphics[width=2.4in]{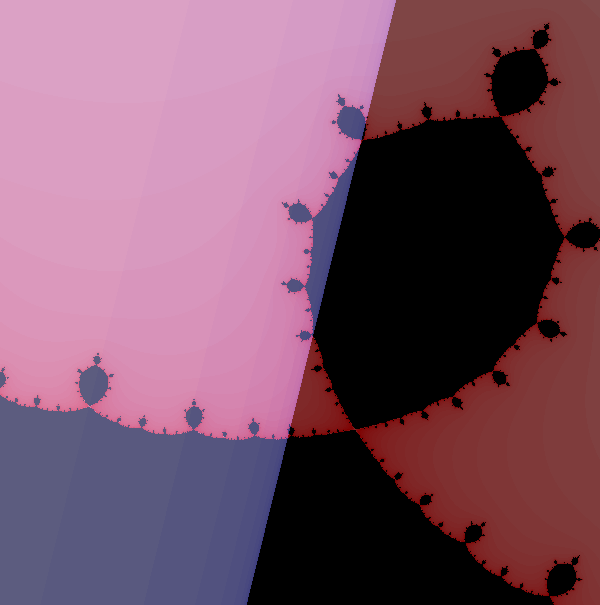}
\caption{ \small Filled Julia sets of $f(z) = z^2$ and $g(z) = z^2 - 1$, superimposed.  At right, a zoom of the intersection of their boundaries, suggesting a possibly infinite overlap of Julia sets.}
\label{circle}
\end{figure}

The proof is elementary and has two steps:
\begin{enumerate}
\item  Show that the subset of unit circle
	$$S = \{e^{2\pi i t}:  t \in [0, 1/30] \cup [1/12, 5/12]\}$$ 
lies in the Fatou set for $g$; and
\item for every $n\geq 3$, the set $S$ contains at least one primitive $n$-th root of unity.
\end{enumerate}
Step (1) follows from a series of simple estimates, examining how $g$ acts on arcs of the unit circle.  Step (2) can be checked by hand by observing that for each $12 < n < 30$, there is some $k$ with $(k, n) = 1$ and $k/n \in [1/12, 5/12]$.

\subsection{The Julia sets are distinct}\label{unique}
It is well known that, for any polynomial, all but finitely many of the periodic points of $f$ will be contained in its Julia set, $J(f)$ is the accumulation set of the (pre)periodic points of $f$, and all of the preperiodic points (other than $\infty$) form a subset of the filled Julia set.  Therefore 
	$$\Preper(f_{c_1}) = \Preper(f_{c_2}) \implies J(f_{c_1}) = J(f_{c_2})$$
for any $c_1, c_2 \in \C$.  But it is also known that the Julia set determines $c$ in this family $f_c(z) = z^2 + c$ \cite[Supplement to Theorem 1]{Baker:Eremenko}, providing the equivalence stated in \eqref{known equivalence}; see also \cite[Theorem 1]{Beardon:symmetries}.

\bigskip
\section{archimedean estimates} \label{archestimates}

In this section, we will carry out some archimedean estimates needed for the proofs of our main theorems. We work with $c\in \C $ and the Euclidean norm $|\cdot|$.  We let $\lambda_c(z)$ denote the escape-rate function of $f_c(z) =z^2+c$, defined by
	$$\lambda_c(z) = \lim_{n\to\infty} \frac{1}{2^n} \log ^+|f_c^n(z)|,$$
where $\log^+ = \max\{\log, 0\}$, and let $\mu_{c}$ denote the corresponding equilibrium measure supported on the Julia set $J_c$.  Where possible, we provide explicit constants in our estimates, even if they are not optimal.

\subsection{Distortion}
We first recall some basic distortion statements for conformal maps.

\begin{theorem} [Koebe 1/4 Theorem]  \label{Koebe}
Suppose $f: \D \to \C$ is univalent with $f(0) = 0$ and $f'(0) = 1$.  Then $f(\D) \supset D(0, 1/4)$.
\end{theorem}

\begin{theorem} \cite[Corollary 3.3]{Branner:Hubbard:1} \label{Koebe at infinity} Let $U_R = \Chat\setminus \overline{D}(0,R)$ and suppose $f: U_R \to \Chat$ is univalent and satisfies 
	$$f(z) = z + \sum_{n\geq 1} \frac{a_n}{z^n}$$
near $\infty$.  Then we have
	$$f(U_R) \supset U_{2R}.$$
\end{theorem}

Applying these theorems to the B\"ottcher coordinate $\phi_c$ near $\infty$ for $f_c(z) = z^2+c$  and to the uniformizing map $\Phi$ for the complement of the Mandelbrot set $\cM$ (see \cite[\S 9]{Milnor:dynamics} or \cite[Chapter {VIII}.3-4]{CarlesonGamelin:book} for more information), we get some simple inequalities.

\begin{prop}  \label{parameter distortion}
For all $c$ with $|c| > 2$ we have 
$$\log|c| - \log 2 \leq \lambda_c(c) \leq \log |c| + \log 2.$$
\end{prop}

\begin{proof}
Let $\Phi(c) = \phi_c(c)$ be the uniformizing map from $\Chat\setminus \mathcal{M}$ to $\Chat \setminus \overline{\D}$ so that $\lambda_c(c) = \log|\Phi(c)|$.  For the lower bound on $\lambda_{c}(c)$, applying Theorem \ref{Koebe at infinity} to $\Phi^{-1}$ and sets $U_R$ with $R\geq 1$ gives
	$$|c| \leq 2 \, e^{\lambda_{c}(c)}$$
for all $c \not\in \cM$, so that
  	$$\lambda_{c}(c) \geq \log|c| - \log 2.$$
For the upper bound on $\lambda_{c}(c)$, apply Theorem \ref{Koebe at infinity} to $\Phi$ and sets $U_{R}$, $R > 2$.  Then $\Phi(U_R) \supset U_{2R}$ implies that 
	$$e^{\lambda_c(c)} \leq 2 |c|,$$
for $|c| = R > 2$.  
\end{proof}

We can do similar things in the dynamical plane.  

\begin{prop}  \label{dynamical distortion}
For each $c$ with $|c| > 2$ and every $z$ with $|z| > 2 e^{\lambda_{c}(0)}$ (so in particular for all $|z| > 2^{3/2}|c|^{1/2}$), 
we have
	$$\log|z| - \log 2 \leq \lambda_{c}(z) \leq \log|z| + \log 2.$$ 
\end{prop}

\begin{proof}
Let $R_0 = e^{\lambda_{c}(0)}$.  Apply Theorem \ref{Koebe at infinity} to $\phi_c^{-1}$ and sets $U_{R}$ for $R\geq R_0$. Then for $z \in \phi_c^{-1}(U_{R_0})$ and $R = e^{\lambda_c(z)}$, we find that 
	$$|z| \leq 2 \, e^{\lambda_{c}(z)}. $$
In particular, the estimate holds for all $|z| > 2 e^{\lambda_{c}(0)}$ because $\phi_c^{-1}(U_{R_0}) \supset U_{2R_0}$.  This gives the lower bound of the proposition.

For the upper bound, set $R' = 2e^{\lambda_{c}(0)} = 2R_0$, so that $\phi_c$ is univalent on $U_{R'}$.  Applying Theorem \ref{Koebe at infinity} to $\phi_c$ on sets $U_R$ for $R \geq R'$, we have 
	$$e^{\lambda_c(z)}\leq 2|z|$$
for each $|z| = R > R'$.  Therefore, 
	$$\lambda_{c}(z) \leq \log|z| + \log 2$$
for all $|z| > 2e^{\lambda_c(0)}$.  

Finally, recall that $\lambda_c(0) \leq \frac12 \log |c| + \frac12 \log 2$ for all $|c| > 2$, from Proposition \ref{parameter distortion}.  Thus, 
	$$2^{3/2}|c|^{1/2} = 2 e^{\frac12 \log|c| + \frac12 \log 2} \geq 2 e^{\lambda_c(0)}$$ 
for all $|c| > 2$.  
\end{proof}

\subsection{Controlling escape rates from below}
We will need both upper and lower bounds on the escape rate $\lambda_{c}$ near the Julia set $J_c$ of $f_c(z) = z^2 + c$.  We begin with an elementary observation.

\begin{lemma}  \label{size 1}
Fix any $c$ with $|c| \geq 25$.  Let $\pm b$ be the two zeroes of $f_c$.  Then 	
	$$\mu_c(D(b, 1)) = \mu_c(D(-b, 1)) = 1/2$$
and 
	$$\lambda_{c}(z) \geq \frac{1}{4} \log |c|$$
for all $z\not\in D(b, 1) \cup D(-b, 1)$.  
\end{lemma}

\begin{proof}
First observe that $b =  i\sqrt{c}$, so that $|b| = |c|^{1/2}$.  Suppose $b + t$ lies on the boundary of $D(b, 1)$, so that $|t| = 1$.  Then 	
	$$f_c(b + t) = 2bt + t^2 = t(2b + t)$$
has absolute value $\geq 2|c|^{1/2} - 1 > |c|^{1/2} + 1$ for $|c| \geq 25$.  In particular, $f_c$ sends $D(b, 1)$ with degree 1 over the union $D(b, 1) \cup D(-b, 1)$. Similarly for $D(-b,1)$, proving the first claim about the measure of each disk.  As the Julia set of $f_c$ is contained in these two disks, we know that $\lambda_c$ is harmonic on the complement of their union.  Under one further iterate, we have
	$$|f_c^2(b+t)| \geq 4|c| - 4|c|^{1/2} + 1 - |c| = 3|c| - 4|c|^{1/2} + 1 \geq 2|c|$$
because $|c| \geq 25$.  From Proposition \ref{dynamical distortion}, we conclude that
	$$\lambda_{c}(b+t) = \frac{\lambda_{c}(f_c^2(b+t))}{4} \geq \frac{1}{4} (\log(2|c|) - \log 2)= \frac{1}{4} \log|c|$$
and similarly for $\lambda_c(-b+t)$ with $|t|=1$.  As $\lambda_c$ is harmonic on $\C\setminus (D(b, 1) \cup D(-b, 1))$, this proves the lemma.
\end{proof}

We now extend the statement of Lemma \ref{size 1} to two further preimages of 0 under $f_c$.  

\begin{lemma} \label{escape control} 
For $n = 1, 2, 3$, and for each $c\in \C$, we let $D_n(c)$ be the union of the $2^n$ disks of radius $\eps_n = |2c|^{-(n-1)/2}$ centered at the solutions $z$ to $f^n_c(z) = 0$.  For each  $|c| \geq 25$, the $2^n$ disks are disjoint, each has $\mu_c$-measure $1/2^n$, and  
$$\lambda_{c}(z) \geq \frac{1}{2^{n+1}} \log |c|$$
for all $z \not\in D_n(c)$ and $n = 1,2, 3$.
\end{lemma}

\begin{proof} 
Lemma \ref{size 1} provides the result for $n =1$ and for any $|c| \geq 25$.  Note that the two disks of radius $2 \eps_1 = 2$ around the solutions to $f(z) = 0$ are disjoint.  

For $n = 2,3$, suppose that $z=w$ is a solution to $f^n_c(z) = 0$.  Note that
	$$\lambda_c(w) = \frac{1}{2^{n+1}} \lambda_c(c) \leq \frac{1}{8}\lambda_c(c) \leq \frac{1}{8}(\log |c| + \log 2) < \frac{1}{4} \log|c|$$
by Proposition \ref{parameter distortion}.  Since $|c| \geq 25$, it follows that the point $w$ must lie in the set $D_1(c)$ by Lemma \ref{size 1}.  In particular, this implies that $|w| > |c|^{1/2} - 1$, so that 
\begin{eqnarray*} 
|f_c(w+t) - f_c(w)| &=&  |t(2w+t)| \\
	&\geq& \frac{1}{|2c|^{(n-1)/2}} \left( 2 |c|^{1/2} - 2 - |2c|^{-(n-1)/2} \right) \\
	&=& \frac{1}{|2c|^{(n-2)/2}} \left( \sqrt{2}\left(1 - \frac{1}{|c|^{1/2}} \right) - \frac{1}{|2c|^{n/2}}\right),
\end{eqnarray*}
for all $t$ with $|t|=\eps_n = |2c|^{-(n-1)/2}$ and each $w$ satisfying $f^n(w) = 0$ with $n = 2$ or $3$.  As $|c| \geq 25$, we have 
$$\sqrt{2}\left(1 - \frac{1}{|c|^{1/2}} \right) - \frac{1}{|2c|^{n/2}} \geq \frac{4\sqrt{2}}{5} - \frac{1}{50} > 1$$
for $n = 2, 3$, and we conclude that 
\begin{equation} \label{lower 23}
 	|f_c(w+t) - f_c(w)| > \frac{1}{|2c|^{(n-2)/2}} = \eps_{n-1}
\end{equation}
for $|t| = \eps_n$ and $f^n(w) = 0$.  By a similar argument, we also have 
\begin{equation} \label{upper 23}
	|f_c(w+t) - f_c(w)| \leq 2 \eps_{n-1}
\end{equation}
for $|t| = \eps_n$ and $f^n(w) = 0$, $n= 2,3$.  

The estimates \eqref{lower 23} and \eqref{upper 23} for $n=2$ imply that the four disks of radius $\eps_2$ are disjoint:  two solutions to $f^2(z)=0$ lie in each component of $D_1(c)$, and the disks around each of these are mapped into disjoint disks of radius 2 around $\pm i \sqrt{c}$, covering the two components of $D_1(c)$.  It follows that the $\mu_c$-measure of each component of $D_2(c)$ is exactly $\frac14$.   Moreover, Lemma \ref{size 1} implies that $\lambda_c(f_c(w+t)) \geq \frac14 \log|c|$ for $|t| = \eps_2$ and $f^2(w)=0$, so that $\lambda_c(w+t) \geq \frac18 \log |c|$.  As $\lambda_c$ is a harmonic function outside of the Julia set, we therefore have
	$$\lambda_c(z) \geq \frac18 \log|c|$$
for all $z \not\in D_2(c)$.

It remains only to show that the four disks centered at the solutions to $f^2(z)=0$ of radius $2\eps_2 = \sqrt{2}/|c|^{1/2}$ are disjoint, for this will imply that the 8 disks of radius $\eps_3$ (centered at the solutions to $f^3(z)=0$) must also be disjoint, from  \eqref{lower 23} and \eqref{upper 23} for $n=3$.  It also immediately follows that each of the 8 components of $D_3(c)$ has $\mu_c$-measure equal to $\frac18$, and moreover that 
	$$\lambda(w+t) \geq \frac{1}{16} \log |c|$$
for $f^3(w) = 0$ and $|t| = \eps_3$, so that 
	$$\lambda_c(z) \geq \frac{1}{16} \log|c|$$
for all $z \not\in D_3$.  

This disjointness is clear for $|c|$ sufficiently large.  Indeed, the points $w$ satisfying $f^2_c(w) = 0$ have the form $\pm \beta(c)$ and $\pm \beta'(c)$ where
\begin{equation} \label{binomial estimate}
     \left|\beta(c) -\left( i \sqrt{c} + \frac{1}{2} + \frac{i}{8\sqrt{c}} \right)\right|=\left|\sum_{j=3}^\infty C_{\frac{1}{2}}^j\cdot \frac{1}{(i\sqrt{c})^{j-1}}\right|\leq \frac{5}{4|c|},
\end{equation}
with binomial coefficients 
	$$C^j_{1/2} = \left( \begin{array}{cc} 1/2 \\ j \end{array} \right),$$
and similarly
   $$\left|\beta'(c) -\left( i \sqrt{c} - \frac{1}{2} + \frac{i}{8 \sqrt{c}}\right)\right|\leq \frac{5}{4|c|}.$$
In particular, the distance between the two closest such roots satisfies
\begin{equation}\label{pre-2-root-distance}
|\beta(c) - \beta'(c)| \geq 1-\frac{5}{2|c|}\geq \frac{3}{4}> 2\cdot \sqrt{2}/|c|^{1/2}=4\eps_2
\end{equation}
for $|c|\geq 25$.  This completes the proof.  
\end{proof}

\subsection{Controlling escape rates from above}
We now provide an upper bound, applying the Distortion Theorems stated above.

\begin{lemma}  \label{shrinking rate}
Fix any $c$ with $|c| \geq 25$.  For each $n\geq 1$ and for all $z\in \C$ with 
	$$\dist(z, J_c) < \frac{1}{5\cdot3^n  |c|^{(n-2)/2}}$$
we have 
	$$\lambda_{c}(z) \leq \frac{1}{2^n} (\log|c| + \log 2) < \frac{1}{2^{n-1}} \log |c|. $$
\end{lemma}

\begin{proof}
The two inverse branches of $f_c$ are univalent on $D(0, |c|)$.  Fix any point $z_0$ in $J_c$.  From Lemma \ref{size 1}, we know that $|z_0| \leq |c|^{1/2} + 1$, so that $f_c$ has two univalent branches of the inverse defined on the disk $D(z_0, |c| - |c|^{1/2} - 1)$ and $|(f_c^n)'(z_0)| \leq 2^n(|c|^{1/2}+1)^n$.  Applying Theorem \ref{Koebe} to the inverse branches of each iterate on these disks about points $z_0 \in J_c$, we find
	$$f_c^{-n}D(z_0, |c| - |c|^{1/2} - 1) \supset D\left(f_c^{-n}(z_0), \frac{|c| - |c|^{1/2} - 1}{4 \cdot 2^n (|c|^{1/2} + 1)^n }\right).$$
From Proposition \ref{dynamical distortion} (and the maximum principle for $\lambda_{c}$), we know that $\lambda_{c}(z) \leq \log|c| + \log 2$ on $D(0, c)$, and therefore 
	$$\lambda_{c}(z) \leq \frac{1}{2^n} (\log|c| + \log 2)$$
on each of these disks of radius $(|c| - |c|^{1/2} - 1)/(4 \cdot 2^n (|c|^{1/2} + 1)^n )$ about points in the Julia set.  Finally, we observe that 
	$$\frac{|c| - |c|^{1/2} - 1}{4 \cdot 2^n (|c|^{1/2} + 1)^n } \geq \frac{|c|(1 - |c|^{-1/2} - |c|^{-1})}{4 \cdot 2^n |c|^{n/2}(1 + |c|^{-1/2})^n} \geq \frac{|c|(19/25)}{4 \cdot 2^n |c|^{n/2} (6/5)^n} \geq \frac{1}{5 \cdot 3^n \, |c|^{(n-2)/2}}$$
for all $|c| \geq 25$.	
\end{proof}

\begin{prop} \label{archimedean epsilon}
Fix $L \geq 27$.  For all $0 < r < 1/4$ and for all $c \in \C$, we have
	$$\lambda_c(z) \leq r \log \max\{|c|, L\}$$
for every $z$ in a neighborhood of radius 
	$$\frac{1}{(\max\{|c|, L\})^{3 \log(1/r)}}$$
around the filled Julia set $K_c$.  
\end{prop}

\begin{proof}
First assume that $|c| > L$.  Note that $J_c = K_c$ in this case.  Lemma \ref{shrinking rate} states that 
	$$\lambda_c(z) \leq \frac{1}{2^{n-1}} \log|c|$$
whenever $\dist(z, J_c) < (5\cdot 3^n |c|^{(n-2)/2})^{-1}$, for any $n\geq 1$.  For $L \geq 27 = 3^3$, we have $5 \cdot 3^n = 15 \cdot 3^{n-1} < L^{1 + (n-1)/3} = L^{(n+2)/3}$.  Therefore, 
	$$5\cdot 3^n |c|^{(n-2)/2} \leq L^{(n+2)/3} |c|^{(n-2)/2} < |c|^{5n/6}.$$
For each $n\geq 3$, we set $r_n = 1/2^{n-1}$, so that $n = \log(1/r_n)/(\log 2) + 1$.  Choose any monotone decreasing function $\kappa$ of $r \in (0,1/4]$ so that 
	$$\kappa(r_n) \geq 5(n+1)/6 = \frac{5}{6} \frac{1}{\log 2} \log(1/r_n) + \frac{5}{3} \approx 1.2 \log(1/r_n) + \frac53$$
for all $n\geq 3$.  Then $|c|^{-\kappa(r_n)} \leq |c|^{-5(n+1)/6} < (5\cdot 3^n |c|^{(n-1)/2})^{-1}$, so that any $z$ satisfying $\dist(z, J_c) < |c|^{-\kappa(r_n)}$ will also satisfy $\lambda_c(z) \leq \frac{1}{2^n} \log |c| = r_{n+1} \log |c|$, for all $n\geq 3$, by Lemma \ref{shrinking rate}.  In particular, we can take $\kappa(r) = 3 \log (1/r)$.  For any $r < 1/4$, we choose $n\geq 3$ so that $r_{n+1} \leq r < r_n$; then $\kappa(r) > \kappa(r_n)$, so $\dist(z, J_c) \leq |c|^{-\kappa(r)}$ implies that  
	$$\lambda_c(z) \leq r_{n+1} \log|c| \leq r \log |c|.$$
This proves the proposition for $|c| > L$.

Now assume $|c| \leq L$.  For $|c|>2$, Proposition \ref{dynamical distortion} implies that if $|z| > 2^{3/2} |c|^{1/2}$, then 
	$$\lambda_c(z) \leq \log |z| + \log 2.$$
Consider the circle of radius $L$.  For all $|c| \leq L$, we have $2^{3/2} |c|^{1/2} \leq 2^{3/2} L^{1/2}  < L$, so that 
\begin{equation} \label{circle L}
	\lambda_c(z) \leq \log L + \log 2,
\end{equation}
for all $2 < |c| \leq L$ and for all $|z| = L$.  But then, fixing $z$, and using the fact that $\lambda_c(z)$ is subharmonic in $c$, we obtain the inequality \eqref{circle L} for all $|c| \leq L$ and all $|z| = L$.  

Furthermore, for all $|c| > 2$ and $|z| \geq 2^{3/2}|c|^{1/2}$, we have the lower bound that 
	$$\lambda_c(z) \geq \log |z| - \log 2 \geq \frac12 \log(2|c|) > 0$$
so that the Julia set is contained in a disk of radius $2^{3/2} |c|^{1/2} \leq 2^{3/2}L^{1/2}$.  On the other hand, for $|c| \leq 2$, it is easy to compute that the filled Julia set lies in a closed disk of radius 2, so we have 
	$$K_c \subset D(0, 2^{3/2}L^{1/2})$$
for all $|c| \leq L$.  In particular, the distance between $K_c$ and the circle of radius $L$ is at least 
	$$L - 2^{3/2}L^{1/2} > 12.$$

For a fixed positive integer $n$ and $|c| \leq L$, suppose $z$ is any point within distance $12/(2L)^n$ of $K_c$.  Let $z_0 \in K_c$ denote the closest point to $z$.  As $|f_c'(z)| = |2z| \leq 2L$ for all $|z| \leq L$, we find that  
	$$|f^n_c(z) - f^n_c(z_0)| \leq (2L)^n|z-z_0| < 12.$$
In other words, $f^n_c(z)$ lies within the circle of radius $L$, so that 
	$$\lambda_c(z) = \frac{1}{2^n} \lambda_c(f^n_c(z))  \leq \frac{1}{2^n} (\log L + \log 2) \leq \frac{1}{2^{n-1}} \log L$$
from \eqref{circle L}, for all $z$ within distance $12/(2L)^n$ of the set $K_c$, and for all $|c| \leq L$.  

Note that $2^8/12 < 27 \leq L$ and $2^4 < L$, and so 
	$$12/(2L)^n \geq 1/(2^{n-8}L^{n+1}) \geq 1/L^{(n-8)/4 + n + 1} =1/L^{\frac{5}{4}n - 1}$$
For each $n\geq 3$, we set $r_n = 1/2^{n-1}$ as before, so that $n = \log(1/r_n) / \log 2 + 1$.  Then
	$$\frac{5}{4} (n+1) -1 = \frac{5}{4 \log 2} \log(1/r_n) + \frac{3}{2} \approx 1.8 \log(1/r_n) + \frac{3}{2}$$
As above, we set $\kappa(r) = 3 \log (1/r)$ for $r \in (0, 1/4]$.  For any $r < 1/4$, we choose $n\geq 3$ so that $r_{n+1} \leq r < r_n$; then $\kappa(r) > \kappa(r_n) > \frac54 (n+1) -1$.  Consequently, for all $z$ within distance $1/L^{3 \log(1/r)}$ of the filled Julia set $K_c$, we have that $z$ lies within distance $12/(2L)^{n+1}$ of $K_c$, and therefore 
	$$\lambda_c(z) \leq \frac{1}{2^n} \log L <  r \log L.$$
This completes the proof of the proposition.
\end{proof}

\bigskip
\section{Bounds on the archimedean pairing}\label{archpairingbounds}

In this section, we provide estimates on the archimedean contributions to the pairing $\< f_{c_1}, f_{c_2}\>$, to obtain a local version of Theorem \ref{pairingbound}.  As in the previous section, we work with $c \in \C$, Euclidean absolute value $|\cdot |$ and archimedean escape-rate function $\lambda_c$.   We let $\mu_{c} = \frac{1}{2\pi} \Delta \lambda_c$ denote the equilibrium measure supported on the Julia set $J_c$.  Where possible, we provide explicit constants, even if they are not optimal, for our estimates of the Euclidean energy
		$$E_\infty(c_1,c_2):= \int \lambda_{c_1} \, d\mu_{c_2} = \int \lambda_{c_2} \, d\mu_{c_1}.$$

\begin{theorem} \label{archimedeanbounds} There exist constants $C, C'>0$ so that 
$$\frac{1}{16} \log^+|c_1 - c_2| - C \leq E_\infty(c_1,c_2) \leq \frac{1}{2} \log^{+}\max \{|c_1|,|c_2| \} + C'$$
for all $c_1, c_2 \in \mathbb{C}$.  Furthermore, there exists $L > 0$ so that if $r := \max\{|c_1|, |c_2|\} \geq L$ and 
$$\frac{3}{r^{1/2}} \leq |c_1 - c_2|,$$
then 
$$\frac{1}{64} \log \max \{|c_1|, |c_2| \} \leq E_\infty(c_1,c_2) .$$
\end{theorem}

\begin{remark} \label{arch constants}
The proof shows that we can take $L = 1000$, $C = \frac{1}{16} \log 2L < 1/2$, and $C' = \log 8$ in Theorem \ref{archimedeanbounds}.
\end{remark}

\subsection{Proof of Theorem \ref{archimedeanbounds}}
Throughout this proof, we will assume for notational convenience that 
	$$r = |c_1| \geq |c_2|.$$  
We proceed by cases, determined by just how close the two parameters are.  In each case, we estimate the value of $\lambda_{c_1}$ on the Julia set $J_{c_2}$.  We prove the second statement first, providing a lower bound on $E_\infty$ when $c_1$ and $c_2$ are not too close, assuming $r = |c_1|$ is sufficiently large.  Then we return to the first statement of the theorem.  \\

{\bf Case 0.} Suppose $|c_2| \leq 25$.  For $|c_2| \leq 2$, it is straightforward to compute that the filled Julia set satisfies $K_{c_2} \subset \overline{D(0,2)}$.  For $2 < |c_2| \leq 25$, Proposition \ref{dynamical distortion} provides a lower bound of 
	$$\lambda_{c_2}(z) \geq \log |z| - \log 2 \geq \frac12 \log(2|c_2|) > 0$$
for $|z| \geq 2^{3/2}|c_2|^{1/2}$.  Therefore, the Julia set of $f_{c_2}$ is contained in a disk of radius $2^{3/2} |c_2|^{1/2} \leq 2^{3/2}5$.  Thus, for all $|c_2| \leq 25$ and $|c_1| > (2^{3/2}\cdot 5+1)^2 \approx 229.3$, Lemma \ref{size 1} implies that $\lambda_{c_1}(z) \geq \frac{1}{4} \log |c_1|$ for all $z \in J_{c_2}$.  This gives
	$$\int \lambda_{c_1} \, d\mu_{c_2} \geq \frac{1}{4} \log |c_1|$$
for $|c_2|\leq 25$ and $|c_1| \geq 230$.\\
 
In the following three cases, we assume that $r = |c_1| \geq |c_2| \geq 25$.  The cases are separated according to the distance $|\sqrt{c_1} - \sqrt{c_2}|$ between the square roots of $c_1$ and $c_2$.  Observe that
$$|\sqrt{c_1} - \sqrt{c_2}| < \frac{3}{2|c_1|} \implies  |c_1 - c_2| < \frac{3}{2|c_1|} (2|c_1|^{1/2}) \leq \frac{3}{|c_1|^{1/2}},$$
so these three cases will complete the proof of the second statement of the theorem.  
\\

{\bf Case 1.} Suppose that for any choice of square roots, we have $|\sqrt{c_1} - \sqrt{c_2}| \geq 2.$  By Lemma \ref{size 1} we have $\lambda_{c_1}(z) \geq \frac{1}{4} \log |c_1|$ for all $z \in J_{c_2},$ so
$$\int \lambda_{c_1}(z) \ d \mu_{c_2} \geq \frac{1}{4} \log |c_1|$$
for $|c_1| \geq |c_2| \geq 25$.  

{\bf Case 2.} Suppose that there is a choice of square roots for which $\frac{2}{r^{1/2}} \leq |\sqrt{c_1} - \sqrt{c_2}| < 2.$  With these choices of square roots, the solutions of $f_{c}^2(z) = 0$ are
$$\beta(c) = i \sqrt{c} + \frac{1}{2} + \frac{i}{8\sqrt{c}} + O\left(\frac{1}{|c|}\right)$$
and
$$\beta'(c) = i \sqrt{c} - \frac{1}{2} + \frac{i}{8 \sqrt{c}} + O\left(\frac{1}{|c|}\right),$$
along with $-\beta(c)$ and $-\beta'(c).$  By Lemma \ref{escape control}, if the disk $D(\beta(c_2), 1/|2c_2|^{1/2})$ does not intersect any disk of radius $1/|2c_1|^{1/2}$ about a solution of $f_{c_1}^2(z) = 0,$ then for all $z \in D(\beta(c_2), 1/|2c_2|^{1/2})$ we have 
$$\lambda_{c_1}(z) \geq \frac{1}{8} \log |c_1|,$$
and since the same is true for the disk centered at $-\beta(c_2)$ by $\pm$ invariance, the inequality is satisfied for a set of $\mu_{c_2}$-measure $1/2$.  Therefore,
$$\int \lambda_{c_1} \ d\mu_{c_2} \geq \frac{1}{16} \log |c_1|.$$
On the other hand, as $ |\sqrt{c_1} - \sqrt{c_2}| < 2$, if $D(\beta(c_2), 1/|2c_2|^{1/2})$ intersects any disk of radius $1/|2c_1|^{1/2}$ about a solution of $f_{c_1}^2(z) = 0,$ that disk must be centered at either $\beta(c_1)$ or $\beta'(c_1),$ since $|\beta(c_2) + \beta(c_1)| \geq |c_1|^{1/2}$ and similarly for $\beta(c_2) + \beta'(c_1).$  We have
$$\beta(c_1) - \beta(c_2) = i(\sqrt{c_1} - \sqrt{c_2}) + \frac{i}{8} \left( \frac{1}{\sqrt{c_1}} - \frac{1}{\sqrt{c_2}} \right) + O\left( \frac{1}{|c_2|} \right),$$
so that using the assumed bounds, we have
$$|\beta(c_1) - \beta(c_2)| \geq \frac{2}{|c_1|^{1/2}} - \frac{1}{8} \left( \frac{4 }{|c_1 c_2|^{1/2}} \right) + O\left( \frac{1}{|c_2|} \right) = \frac{2}{|c_1|^{1/2}} + O \left( \frac{1}{|c_2|} \right),$$
using for the middle term the crude bound $|c_1 - c_2| \leq 4|c_1|^{1/2}$ implied by $|\sqrt{c_1} - \sqrt{c_2}| < 2$.  Then, exactly as in \eqref{binomial estimate} in the proof of Lemma \ref{escape control}, we can take 
   $$|\beta(c_1) - \beta(c_2)|\geq \frac{2}{|c_1|^{1/2}}-\frac{5}{2}\frac{1}{|c_2|},$$
because $|c_1| \geq |c_2| \geq 25$.  Since $|c_2|^{1/2} > |c_1|^{1/2} - 2$,  taking $|c_1| \geq 230$ is enough to guarantee this distance will be larger than $2(1/|2c_2|^{1/2})$, and the disks $D(\beta(c_1), 1/|2c_1|^{1/2})$ and $D(\beta(c_2), 1/|2c_2|^{1/2})$ will be disjoint.   Similarly we deduce that the disks $D(\beta'(c_1), 1/|2c_1|^{1/2})$ and $D(\beta'(c_2), 1/|2c_2|^{1/2})$ are disjoint.

But observe also that if
$$|\beta(c_2) - \beta'(c_1)| < \frac{2}{|2c_2|^{1/2}} = \frac{\sqrt{2}}{|c_2|^{1/2}},$$
then $\beta'(c_2)$ must be far from both $\beta'(c_1)$ and $\beta(c_1)$, because 
\begin{eqnarray*} |\beta'(c_2) - \beta(c_1)| &=& |\beta'(c_2) - \beta(c_2) +  \beta(c_2) -\beta'(c_1)+ \beta'(c_1) - \beta(c_1)| \\
&\geq& |\beta'(c_2) -  \beta(c_2) + \beta'(c_1) - \beta(c_1)| - \frac{\sqrt{2}}{|c_2|^{1/2}} \\
&=& 2 - \frac{\sqrt{2}}{|c_2|^{1/2}} -4\cdot \frac{5}{4}\frac{1}{|c_2|}.
\end{eqnarray*}
We therefore have, for $r = |c_1|\geq 230$ and square roots satisfying $\frac{2}{r^{1/2}} \leq |\sqrt{c_1} - \sqrt{c_2}| < 2$, at least one of the four disks of radius $1/|2c_2|^{1/2}$ around a solution to $f_{c_2}^2(z)=0$ is disjoint from the four disks of radius $1/|2c_1|^{1/2}$ about the four solutions of $f_{c_1}^2(z) = 0$.  By the $\pm$ symmetry, two of these disks must be disjoint.  As these two disks carry $1/2$ of the measure $\mu_{c_2}$, we have by Lemma \ref{escape control} that 
$$\int \lambda_{c_1} \ d \mu_{c_2} \geq \frac{1}{16} \log |c_1|.$$

{\bf Case 3.} Suppose there is a choice of square roots for which 
$$\frac{3}{2r} \leq |\sqrt{c_1} - \sqrt{c_2}| < \frac{2}{r^{1/2}}.$$
We will argue precisely as in Case 2, but with the third preimages of 0 rather than second.  Two solutions of $f_{c}^3(z) = 0$ have the form 
$$s(c) := i \sqrt{c} + \frac{1}{2} - \frac{i}{8 \sqrt{c}} + \frac{1}{8c} + O \left( \frac{1}{|c|^{3/2}} \right)$$
and
$$s'(c) := i \sqrt{c} + \frac{1}{2} + \frac{3i}{8 \sqrt{c}} - \frac{1}{8c} + O \left( \frac{1}{|c|^{3/2}} \right).$$
From the Taylor expansion, and the fact that $|c|>100$, the above big-O's satisfy the following estimate, to be proved below:
\begin{equation} \label{more delicate}
  \left|s(c) - \left(i \sqrt{c} + \frac{1}{2} - \frac{i}{8 \sqrt{c}} + \frac{1}{8c}\right) \right|\leq 5\frac{1}{|c|^{\frac{3}{2}}}
\end{equation}
and similarly for $s'(c)$.  Notice that under the action of $f_c$, we have $s(c) \mapsto \beta(c)$ and $s'(c) \mapsto \beta'(c),$ and that both $s(c)$ and $s'(c)$ are distance at least 1/2 from all other solutions of $f^3_c(z)$ (except each other).  

If the disk of radius $1/|2c_2|$ about $s(c_2)$ intersects any disk of radius $1/|2c_1|$ about a solution of $f_{c_1}^3(z) = 0$, then that disk must be centered at either $s(c_1)$ or $s'(c_1)$, because of the form of the power series expansions of the various third preimages of 0.  If this disk $D(s(c_2), 1/|2c_2|)$ is disjoint from both $D(s(c_1), 1/|2c_1|)$ and $D(s'(c_1), 1/|2c_1|)$, then from the $\pm$ symmetry and Lemma \ref{escape control}, we have
$$\int \lambda_{c_1} \ d \mu_{c_2} \geq \frac{1}{4 \cdot 16} \log |c_1| = \frac{1}{64} \log |c_1|.$$ 

Now, we have by our assumed bounds that $|\sqrt{c_1} - \sqrt{c_2}| < 2 |c_1|^{-1/2}$, so that 
	$$|c_1 - c_2| = |(\sqrt{c_1}-\sqrt{c_2})(\sqrt{c_1}+\sqrt{c_2})| < \frac{4|c_1|^{1/2}}{|c_1|^{1/2}} = 4,$$
and therefore, 
$$|s(c_1) - s(c_2)| \geq   \frac{3}{2|c_1|} - \frac{2}{8|c_2|^{3/2}}- \frac{10}{|c_2|^{\frac{3}{2}}} > \frac{1}{|c_2|}$$
for $|c_1|\geq 1000$.  So the disks $D(s(c_1), 1/|2c_1|)$ and $D(s(c_2), 1/|2c_2|)$ are disjoint.  But if 
$$|s'(c_1) - s(c_2)| <  \frac{1}{|c_2|},$$
then
\begin{eqnarray*} |s(c_1) - s'(c_2)| &=& |s(c_1) - s'(c_1) + s'(c_1) - s(c_2) + s(c_2) - s'(c_2)| \\
&\geq& |s(c_1) - s'(c_1) + s(c_2) - s'(c_2)| - \frac{1}{|c_2|} \\
& \geq&\left| \frac{-i}{2} \left( \frac{1}{\sqrt{c_1}} + \frac{1}{\sqrt{c_2}}\right) +  2O \left( \frac{1}{|c_2|^{3/2}} \right) \right|  - \frac{1}{|c_2|} \\
& =&\left| \frac{-i}{2} \left( \frac{\sqrt{c_1} + \sqrt{c_2}}{\sqrt{c_1c_2}}\right) +  2O \left( \frac{1}{|c_2|^{3/2}} \right) \right|  - \frac{1}{|c_2|} \\
&\geq&\left| \frac{-i}{2} \left( \frac{\sqrt{c_1} + \sqrt{c_2}}{\sqrt{c_1c_2}}\right)\right| - \frac{10}{|c_2|^{\frac{3}{2}}}- \frac{1}{|c_2|}\\
&\geq & \frac{1}{2|c_1|^{1/2}} > \frac{1}{|c_2|}
\end{eqnarray*}
for $|c_1|\geq 1000$.  We conclude in this case that the disk $D(s'(c_2), 1/|2c_2|)$ is disjoint from the eight disks of radius $1/|2c_1|$ about solutions of $f_{c_1}^3(z) = 0$, and hence (again using symmetry and Lemma \ref{escape control}) we have
$$\int \lambda_{c_1} \ d \mu_{c_2} \geq \frac{1}{64} \log |c_1|.$$

\bigskip\noindent{\bf Proof of estimate \eqref{more delicate}.}
From the estimate \eqref{binomial estimate}, we have 
   $$\beta=i\sqrt c-\frac{1}{2}+\frac{i}{8\sqrt{c}}+a$$
with $|a|\leq 5/4|c|$ whenever $|c|\geq 25$. Furthermore, let us assume that 
$$s=\sqrt{-c+\beta}=i\sqrt{c}\left(1+\frac{1}{i\sqrt{c}}+\frac{1}{2c}+\frac{1}{8i\sqrt c^3}-\frac{a}{c}\right)^{1/2}.$$
For convenience, we set 
   $$b=\left(1+\frac{1}{i\sqrt{c}}+\frac{1}{2c}+\frac{1}{8i\sqrt c^3}-\frac{a}{c}\right)^{1/2} \;\mbox{ and }\;
   e=\frac{1}{i\sqrt{c}}+\frac{1}{2c}+\frac{1}{8i\sqrt c^3}-\frac{a}{c}$$
and then one has 
\begin{equation}\label{b expansion} b=(1+e)^{1/2}=1+\frac{1}{2}e-\frac{1}{8}e^2+\frac{1}{16}e^3+\sum_{n\geq 4}C_{1/2}^ne^n
\end{equation}
where $C_{1/2}^n$ are the binomial coefficients. In the following, we assume that $|c|\geq 100$, 
so that $e$ can be estimated as $|e|\leq \frac{11}{10}\frac{1}{\sqrt{|c|}}$. Consequently as $|C_{1/2}^n|<1$, we have
   $$\left| \sum_{n\geq 4}C_{1/2}^ne^n \right|\leq 1.7\frac{1}{|c|^2}\quad\textup{ and }\quad \frac{1}{2}\left| e-\left(\frac{1}{i\sqrt{c}}+\frac{1}{2c}+\frac{1}{8i\sqrt c^3}\right)\right|=\left|
   \frac{a}{2c}\right|\leq \frac{5}{8|c|^2},$$
and moreover
   $$\frac{1}{8}\left|e^2-\left(-\frac{1}{c}+\frac{1}{i\sqrt{c}^3}\right) \right|\leq \frac{1}{4}{|c|^2} \quad\textup{ and }\quad \frac{1}{16}\left|e^3-\left(-\frac{1}{i\sqrt{c}^3}\right) \right|\leq \frac{1}{4|c|^2}.$$
Finally, we get an estimate on $b$ using the expansion \eqref{b expansion} and therefore the estimate \eqref{more delicate} on $s$ since $s=i\sqrt{c}\cdot b$.  This completes the proof of \eqref{more delicate}.

\bigskip
We are now ready to prove the first statement of the theorem.  Choose any $L\geq 1000$.  If  $|c_1 - c_2| \leq 2L$, then the lower bound on $E_\infty$ holds trivially with the constant $ \frac{1}{16} \log 2L$.  In particular, it holds whenever $\max \{ |c_1|, |c_2| \} \leq L$.

Now suppose that $|c_1 - c_2| \geq \max \{ |c_1|, |c_2| \} > L$.  Then the hypotheses of either Case 0 or 1 hold, and we have
$$\frac{1}{8} \log^+|c_1 - c_2| \leq \frac{1}{4} \log |c_1| \leq \int \lambda_{c_1}(z) \ d \mu_{c_2},$$
as needed.  On the other hand, if $\max \{ |c_1|, |c_2| \} > L$ and $2L < |c_1 - c_2| < \max \{ |c_1|, |c_2| \}$, then the hypotheses of either Case 0, 1, or 2 hold, and we have
$$\frac{1}{16} \log^+|c_1 - c_2| \leq \frac{1}{16} \log^+ \max \{ |c_1|, |c_2| \} \leq \int \lambda_{c_1} \ d \mu_{c_2}.$$
Thus, we have proved the lower bound in the first statement of the theorem,
$$\frac{1}{16} \log^+|c_1 - c_2| - C \leq \int \lambda_{c_1} \ d \mu_{c_2}$$
for all $c_1, c_2\in \C$, with $C = \frac{1}{16} \log 2L.$ 

To prove the upper bound, suppose first that $|c_1| = \max \{ |c_1|, |c_2| \} > 2.$  For $|c_2| \geq 2,$ by Proposition \ref{dynamical distortion}, the Julia set of $f_{c_2}$ is contained in the disk $D(0, 2^{3/2}|c_2|^{1/2})$.  For $|c_2|\leq 2$, we have $J_{c_2} \subset \overline{D(0, 2)}$.  By Proposition \ref{dynamical distortion}, we have by the Maximum Principle that
\begin{equation} \label{on a circle}
\lambda_{c_1}(z) \leq \frac{3}{2} \log 2 + \frac{1}{2} \log |c_1| +\log 2
\end{equation}
for all $z \in D(0, 2^{3/2}|c_1|^{1/2})$ (which contains $J_{c_2}$).  

On the other hand, for $|c_1| = \max \{ |c_1|, |c_2| \} \leq 2$, we use the fact that $\lambda_{c_1}(z)$ is subharmonic in both $z$ and $c_1$, so that the inequality \eqref{on a circle} holds on the circle $\{|z| = 4\}$, replacing $|c_1|$ with 2, for all $|c_1| \leq 2$.  

Applying this inequality to $z\in J_{c_2}$ we see that 
$$\int \lambda_{c_1} \ d \mu_{c_2} \leq \frac{1}{2} \log^+ \max \{|c_1|, |c_2| \} + \log 8$$
for all $c_1, c_2 \in \C$.  
This completes the proof of Theorem \ref{archimedeanbounds}.

\bigskip
\section{Nonarchimedean bounds for prime $p \ne 2$} \label{not2p}

Let $c_1 \not= c_2$ be two elements of $\Qbar$.  Fix a number field $K$ containing $c_1$ and $c_2$, and fix a non-archimedean place $v$ of $K$ which does {\em not} lie over the prime $p=2$.  Let $K_v$ denote the completion of $K$ with respect to $|\cdot|_v$, and let $\C_v$ denote the completion of an algebraic closure of $K$.  In this section, we provide estimates on  the local energy
	$$E_v:= \int \lambda_{c_1, v} \, d\mu_{2,v} = \int \lambda_{c_2,v} \, d\mu_{1,v}.$$
Because the place $v$ is fixed throughout this section, we will drop the dependence on $v$ in the absolute value $|\cdot|_v$, denote the local Julia set of $f_c$ (in the Berkovich affine line $\A^{1,an}_v$ defined over $\C_v$) by $J_{c}$, its escape rate by $\lambda_c$, and the equilibrium measure by $\mu_c$.

\begin{theorem} \label{not 2 bound} Fix a number field $K$ and place $v$ of $K$ that does not divide the prime $p=2$.  For all $c_1, c_2 \in K$, we have 
$$\frac{1}{4} \log^+|c_1 - c_2|  \leq E_v \leq \frac{1}{2} \log^{+}\max \{|c_1|,|c_2| \}.$$
Furthermore, if $r := |c_1| = |c_2| > 1$ and 
$$ |c_1 - c_2| > \frac{1}{r^{1/2}},$$
then 
$$E_v \geq \frac{1}{16} \log r .$$
\end{theorem}

We also prove an estimate on $\lambda_c$ from above, at points near the $v$-adic Julia set of $f_c$, that will be needed for the proof of Theorem \ref{boundingN}.

\subsection{Structure of the Julia set}
We work with the dynamics of $f_c$ on the Berkovich affine line $\A^{1,an}_v$, associated to the complete and algebraically closed field $\C_v$, and we denote by $\zeta_{x,r}$ the Type II point corresponding to the disk of radius $r \in \Q_{>0}$ about $x$.  We refer to \cite[Chapter 8]{Benedetto:book} for more information about the Julia set on the Berkovich affine line, and to the article \cite{Benedetto:Briend:Perdry} for more information about the Julia sets of quadratic polynomials.

For $|c| \leq 1$, the map $f_c$ has good reduction, so that $J_c = \zeta_{0,1}$ is the Gauss point and $\lambda_c(z) = \log^+|z|$.  For $|c| > 1$, the Julia set of $f_c$ is a Cantor set of Type I points, lying in the union of the two open disks $D(\pm b, |c|^{1/2})$ with $f_c(\pm b) = 0$.  In particular, all points $z \in J_{c,v}$ will satisfy $|z| = |c|^{1/2}$.  For any point $z$ with absolute value $|z|> |c|^{1/2}$, we have $|f^n(z)| = |z|^{2^n}$ for all $n\geq 1$, so that 
\begin{equation} \label{lambda value5}
	\lambda_c(z) = \log|z| \quad\mbox{for}\quad |z|> |c|^{1/2}
\end{equation}
and
\begin{equation} \label{upper bound on lambda5}
	\lambda_c(z) \leq \frac12 \log|c| \quad\mbox{for}\quad |z| \leq |c|^{1/2}.
\end{equation}
Taking one further preimage of 0, we may choose $\beta$ and $\beta'$ so that 
\begin{equation} \label{beta definition5}
	f_c(\beta) = b, \quad  f_c(\beta') = -b, \quad |\beta - b| = |\beta'-b| = |\beta - \beta'| = 1,
\end{equation}
and the Julia set will lie in the union of the four disks $D(\pm \beta, 1)$ and  
$D(\pm \beta', 1)$.  
See Figure \ref{nonarch Julia5}.

\begin{figure} [h]
\includegraphics[width=3.5in]{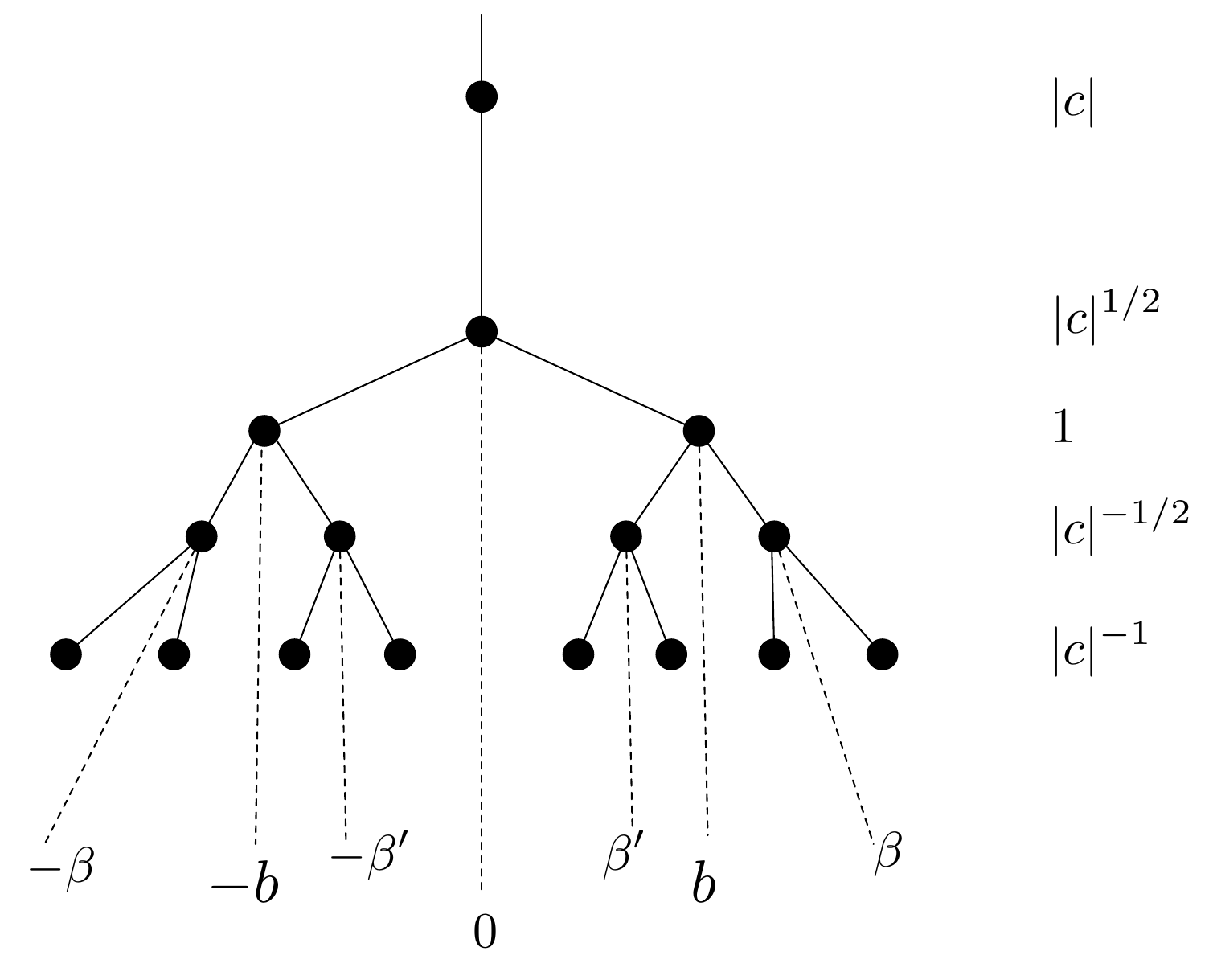}
\caption{ \small The tree structure of the non-archimedean Julia set, with $|c|_v >1$ and $v\not| ~2$.}
\label{nonarch Julia5}
\end{figure}

We will repeatedly exploit the symmetry of the Julia set $J_c$ during the proof of Theorem \ref{not 2 bound}.  For example, identifying the branches from the Type II point $\zeta_{b, 1}$ with the elements of $\P^1(\overline{\mathbb{F}}_p)$ (where we always identify the branch containing $\infty$ as $\infty\in \P^1(\overline{\mathbb{F}}_p)$), and denoting the class of $z \in \C_v$ by $\tilde{z}$, we have 
\begin{equation} \label{linear symm}
	\tilde{\beta} = \tilde{b} + \alpha \quad \mbox{ and } \tilde{\beta}' = \tilde{b} - \alpha
\end{equation}
for some $\alpha \in \overline{\mathbb{F}}_p$, because the transformation from $\zeta_{b,1}$ to its image $f_c(\zeta_{b,1}) = \zeta_{0, |c|^{1/2}}$ is affine in these local coordinates.  In other words, the disks containing the Julia set are centered around the preimages of 0.  The same symmetry holds for the iterated preimages of $\zeta_{b,1}$ and $\zeta_{-b,1}$; the branches containing the Julia set will be symmetric about the preimages of 0, independent of the choice of coordinates, because the iterated map to $\zeta_{0,|c|^{1/2}}$ is affine.

For the proof of Theorem \ref{not 2 bound}, it is also important to keep in mind how distances scale under iteration.  For all $x \in J_c$ and all $z = x+y$ with $|y| < |c|^{1/2}$, we have 
\begin{equation} \label{general p scale}
	|f_c(z) - f_c(x)|_v = |2x y + y^2 |_v = |y| |c|^{1/2}.
\end{equation}

\subsection{Proof of Theorem \ref{not 2 bound}}

If $|c_1|$ or $|c_2|$ is $\leq 1$, then because of good reduction, we have 
	$$E_v = \frac12 \max\{\log^+|c_1|, \log^+|c_2|\} = \frac12 \log^+|c_1-c_2|.$$

If $|c_1|$ and $|c_2|$ are both $>1$, then we can split into further cases.  For $|c_1| > |c_2|$, we have 
	$$\lambda_{c_2}(z_1) = \frac12 \log |c_1|$$
from \eqref{lambda value5} for all points $z_1$ in the Julia set $J_{c_1}$.  Similarly for $|c_1| < |c_2|$, and therefore, 
	$$E_v = \frac12 \max\{\log^+|c_1|, \log^+|c_2|\} = \frac12 \log^+|c_1-c_1|.$$

For the remainder of the proof we assume that 
	$$r := |c_1|=|c_2|>1.$$  
From \eqref{upper bound on lambda5}, we will have 
	$$\lambda_{c_2}(z_1) \leq \frac12 \log|c_2|$$
at all points $z_1$ of the Julia set $J_{c_1}$.   Therefore, 
	$$E_v \leq \frac12 \log |c_2|= \frac12 \log r,$$
proving the upper bound in the theorem.  

For the lower bound on $E_v$, we now break the proof into cases, depending on how close the two parameters are to one another.  

\medskip
{\bf Case 1.}  Assume that 
	$$1 < r^{1/2} < s:= |c_1 - c_2| \leq r = |c_1|= |c_2|.$$
Let $z_1$ be any point in the Julia set $J_{c_1}$.   Then its image $z_1^2 + c_1$ must lie in one of the disks $D(\pm b_1, r^{1/2})$, where $f_{c_1}(\pm b_1)=0$, and have absolute value $r^{1/2}$, so that $f_{c_2}(z_1) = z_1^2 + c_2 = (z_1^2 + c_1) + (c_2 - c_1)$ satisfies 
	$$|f_{c_2}(z_1)| = s > r^{1/2}.$$  
It follows that $|f_{c_2}^n(z_1)| = s^{2^{n-1}}$ for all $n$.  This gives
	$$\lambda_{c_2}(z_1) = \frac12 \log s = \frac12 \log|c_1-c_2|$$
for all $z_1$ in the Julia set of $f_{c_1}$.  Therefore, 
	$$E_v = \frac12 \log |c_1 - c_2| > \frac14 \log r.$$

\medskip
{\bf Case 2.}
Now suppose $|c_1-c_2| = r^{1/2}$, and recall that $b_i^2 = -c_i$, for $i = 1,2$.  Note that 
\begin{equation} \label{b product}
	(b_1 + b_2)(b_1-b_2) = b_1^2 - b_2^2 =  c_2 - c_1
\end{equation}
and at least one of the factors on the left hand side has absolute value $r^{1/2}$ so the other must have absolute value 1.  Let's assume that 
	$$|b_1 - b_2| = 1.$$
If the two branches from $\zeta_{b_1, 1} = \zeta_{b_2,1}$ containing $J_{c_1}$ are disjoint from those containing $J_{c_2}$, then for any element $z_2 \in J_{c_2}$ we have 
	$$|f_{c_1}(z_2)| = r^{1/2} \mbox{ and } |f_{c_1}^n(z_2)| = (r^{1/2})^{2^{n-1}} \mbox{ for all } n\geq 2$$
so that 
	$$\lambda_{c_1}(z_2) = \frac14 \log r = \frac12 \log|c_1-c_2|$$
for all $z_2 \in J_{c_2}$, and 
	$$E_v =  \frac14 \log r = \frac12 \log|c_1-c_2|.$$
	
However, it can happen that one of the branches from $\zeta_{b_1,1}$ intersecting $J_{c_1}$ does coincide with a branch intersecting $J_{c_2}$.  Note that from \eqref{beta definition5}, we have
\begin{equation} \label{beta product}
(\beta_1 - \beta_2)(\beta_1+\beta_2) = \beta_1^2 - \beta_2^2 = b_1 - c_1 - (b_2 - c_2) = (b_1 - b_2) + (c_2 - c_1),
\end{equation}
and the right-hand-side has absolute value $|c_1-c_2| = r^{1/2}$, so that 
	$$|\beta_1 - \beta_2| = 1.$$
But we could have $D(\beta_1, 1) = D(\beta_2', 1)$.  Indeed, 
	$$(\beta_1 - \beta_2')(\beta_1+ \beta_2') = (b_1 + b_2) + (c_2 - c_1)$$
and the terms on the right-hand-side might cancel to give absolute value smaller than $r^{1/2}$.  
But by the symmetry of the disks around the points $b_i$, as explained in \eqref{linear symm}, if $D(\beta_1, 1) = D(\beta_2', 1)$, then the other disks $D(\beta_1',1)$ and $D(\beta_2,1)$ must be disjoint.  Indeed, if $\tilde{b}_1+\alpha_1 = \tilde{\beta_1} = \tilde{\beta_2}'= \tilde{b}_2 - \alpha_2$ and $\tilde{b}_1-\alpha_1 =\tilde{\beta_1}' = \tilde{\beta_2} =   \tilde{b}_2 + \alpha_2$ in $\overline{\mathbb{F}}_p$, then 
	$$2\alpha_1 = -2\alpha_2 \implies \alpha_1 = -\alpha_2 \mbox{ because } p \not=2,$$
so we must have $\tilde{b}_1 = \tilde{b}_2$, which contradicts the fact that $|b_1 - b_2| = 1$.  

It follows that for all $z_2 \in D(\beta_2',1)$, one has
	$$\lambda_{c_1}(z_2) = \frac14 \log r.$$
By the symmetry of the Julia sets, this will also hold for points in the disk $D(-\beta_2',1)$, and together they make up half (with respect to the measure $\mu_{c_2}$) of $J_{c_2}$.  Therefore, 
	$$E_v \geq \frac18 \log r = \frac14 \log|c_1-c_2|.$$

\medskip
{\bf Case 3.}  Assume that 
	$$1 < s:= |c_1 - c_2| < r^{1/2}.$$
Then from \eqref{b product}, we can choose $b_1$ and $b_2$ so that
	$$\frac{1}{r^{1/2}} < |b_1-b_2| = \frac{s}{r^{1/2}} < 1.$$
Also, from \eqref{beta product}, we see that 
	$$\frac{1}{r^{1/2}} < |\beta_1 - \beta_2| = \frac{s}{r^{1/2}} < 1$$
and similarly for $\beta_1'$ and $\beta_2'$.  Consequently, the four disks $D(\pm \beta_1, s/r^{1/2})$ and $D(\pm \beta_1', s/r^{1/2})$ are disjoint from the corresponding disks around $\pm \beta_2$ and $\pm \beta_2'$.  Thus, for any $z_2 \in J_{c_2}$, we have 
	$$\inf_{z_1 \in J_{c_1}} |z_2 - z_1| = s/r^{1/2} \quad\mbox{ and } \quad \inf_{z_1 \in J_{c_1}} |f_{c_1}(z_2) - z_1| = s,$$
and therefore
	$$|f_{c_1}^2(z_2)| = sr^{1/2} \quad \mbox{ and } \quad |f_{c_1}^n(z_2)| = (sr^{1/2})^{2^{n-2}} \mbox{ for all } n>2.$$
This gives 
	$$\lambda_{c_1}(z_2) = \frac14 \log(s r^{1/2})$$
for all $z_2 \in J_{c_2}$, and consequently, 
	$$E_v = \frac14 \log(s r^{1/2}) = \frac18 \log r + \frac14 \log |c_1-c_2|.$$  
In particular, we have 
	$$E_v \geq \frac14 \log |c_1-c_2|$$
in this case, completing the proof of the first statement of the theorem.  
		
\medskip
{\bf Case 4.}
Now suppose $|c_1-c_2| = 1$.  The proof here is similar to Case 2, but we work with the disks around $\beta$ and $\beta'$.  From \eqref{b product} and \eqref{beta product} we can choose our preimages of 0 so that
	$$|b_1 - b_2| = |\beta_1 - \beta_2| = |\beta_1'-\beta_2'| = \frac{1}{r^{1/2}}.$$
Let $\gamma_i$ and $\gamma_i'$, for $i=1,2$, denote further preimages of 0, so that $f_{c_i}^3(\gamma_i) = f_{c_i}^3(\gamma_i') = 0$, chosen so that 
\begin{equation} \label{gamma def}
	|\gamma_i - \beta_i| =  |\gamma_i' - \beta_i| = 1/r^{1/2}
\end{equation}
for $i = 1,2$.  Because of the symmetry of the Julia set $J_{c_i}$ around $\beta_i$, for $i = 1,2$,  as explained in \eqref{linear symm}, if for example the disks $D(\gamma_1, 1/r^{1/2})$ and $D(\gamma_2, 1/r^{1/2})$ coincide, then the disks $D(\gamma_1', 1/r^{1/2})$ and $D(\gamma_2', 1/r^{1/2})$ must be disjoint, because $|\beta_1 - \beta_2| = 1/r^{1/2}$.   Similarly for the disks $D(\gamma_1, 1/r^{1/2})$ and $D(\gamma_2', 1/r^{1/2})$, and also for the disks intersecting the Julia sets near $- \beta_i$ and $\pm \beta_i'$.  

It follows that 
	$$\inf_{z_1 \in J_{c_1}}|f_{c_1}(z)- z_1| = 1, \quad  |f_{c_1}^2(z)| = r^{1/2}, \quad \mbox{ and }  |f_{c_1}^n(z)| = (r^{1/2})^{2^{n-2}} \mbox{ for all } n > 2,$$
for at least half of the points $z$ in $J_{c_2}$.  Therefore
	$$\lambda_{c_1}(z) = \frac18 \log r$$
for at least half of $J_{c_2}$, and consequently, 
	$$E_v \geq \frac{1}{16} \log r$$
in all cases with $|c_1 - c_2| = 1$.

\medskip
{\bf Case 5.}  The final case to treat is with 
	$$1/r^{1/2} < s:= |c_1 - c_2| < 1.$$
We can choose preimages $b_1$ and $b_2$ of 0 so that
	$$\frac1r < |b_1 - b_2| = |\beta_1 - \beta_2| = |\beta_1' - \beta_2'| = \frac{s}{r^{1/2}} < \frac{1}{r^{1/2}}$$
from \eqref{b product} and \eqref{beta product}.  Passing to 3rd preimages of 0, as defined by \eqref{gamma def}, we have 
	$$(\gamma_1 - \gamma_2)(\gamma_1 + \gamma_2) = \gamma_1^2 - \gamma_2^2 = (f_{c_1}(\gamma_1) - f_{c_2}(\gamma_2)) + (c_2- c_1), $$
and similarly for $\gamma_i'$.  Thus, they can be chosen so that 
	$$|\gamma_1 - \gamma_2| = |\gamma_1' - \gamma_2'| = s/r^{1/2} > 1/r.$$
Consequently, all points $z_2 \in J_{c_2}$ will satisfy 
	$$\inf_{z_1 \in J_{c_1}} |z_2 - z_1| = s/r^{1/2}$$
and so
	$$|f_{c_1}^3(z_2)| = r^{3/2} \frac{s}{r^{1/2}} = rs \quad \mbox{ and } |f_{c_1}^n(z_2)| = (rs)^{2^{n-3}} \mbox{ for all } n > 3.$$
Therefore
	$$\lambda_{c_1}(z_2) = \frac18 \log(rs)$$
for all points $z_2\in J_{c_2}$, and 
	$$E_v = \frac18 \log(rs)  > \frac{1}{16} \log r.$$
This completes the proof of the theorem.

\subsection{An upper bound on the local height near the Julia set}
We will use the following proposition in the proof of Theorem \ref{boundingN}.  This is a non-archimedean analog to the distortion estimate provided in Proposition \ref{archimedean epsilon}.

\begin{prop} \label{escapebound pnot2} Suppose $v$ is a non-archimedean place of $K$, not dividing 2.  For each $c$ with $|c| > 1$ and all $0 < r < 1$, we have 
	$$\lambda_{c}(z) \leq r \log |c|_v$$
for all $z$ within distance 
	$$\frac{1}{|c|^{\log(1/r)}}$$
of the Julia set $J_c$ in ${\bf P}^1_v$.  For $|c| \leq 1$, we have $\lambda_{c}(z) = 0$ for all $|z|_v \leq 1$.
\end{prop}

\begin{proof} 
Recall that all points $x$ of the Julia set $J_c$ satisfy $|x| = |c|^{1/2}_v$.  For all $x \in J_c$ and all $z = x+y$ with $|y| < |c|^{1/2}$, we have 
$$|f_c(z) - f_c(x)|_v = |2x y + y^2 |_v = |y| |c|^{1/2}.$$
Recall that $\lambda_c(z) = \log |z|$ for all $|z| > |c|^{1/2}$.   

For any $n\geq 1$ and $\frac{1}{2^n} \leq r < \frac{1}{2^{n-1}}$, we have 
	$$\log\left(\frac1r\right) > (n-1) \log 2 >  \frac{n}{2} - 1.$$
So, for any point $z$ within distance $1/|c|^{\log(1/r)}$ of the Julia set $J_c$, it is also within distance $|c|/|c|^{n/2}$ of the $J_c$, so that we will have 
	$$\lambda_c(z) = 2^{-n} \lambda_c(f^n(z)) \leq 2^{-n} \log |c| \leq r \log |c|.$$  

The proof of the last statement of the proposition is immediate, because $f_c$ has good reduction with $J_c = \zeta_{0,1}$ and $\lambda_{c,v}(z) = \log^+|z|_v$.
\end{proof}

\bigskip
\section{Nonarchimedean bounds for prime $p = 2$} \label{2p}

Let $c_1 \not= c_2$ be two elements of $\Qbar$.  Fix a number field $K$ containing $c_1$ and $c_2$, and fix a non-archimedean place $v$ of $K$ which lies over the prime $p=2$.  We assume that $|\cdot|_v$ is normalized so that $|2|_v = \frac12$.  In this section, we provide estimates on the local energy
	$$E_v:= \int \lambda_{c_1, v} \, d\mu_{c_2,v} = \int \lambda_{c_2,v} \, d\mu_{c_1,v}.$$
Because the place $v$ is fixed throughout this section, we will drop the dependence on $v$ in the absolute value $|\cdot|_v$, denote the local Julia set of $f_c$ by $J_{c}$, its escape rate by $\lambda_c$, and the equilibrium measure by $\mu_c$.

\begin{theorem} \label{p2} 
Suppose $c_1$ and $c_2$ lie in a number field $K$, and $v$ is a non-archimedean place of $K$ with $v \mid 2.$    For all $c_1, c_2 \in K$, we have 
$$\frac{1}{16} \log^+ |c_1 - c_2| - \frac14 \log 2 ~\leq~ E_v ~\leq~ \frac{1}{2} \log^+ \max \{|c_1|, |c_2| \}.$$
Furthermore, if $r := |c_1| = |c_2| > 16$ and 
$$ |c_1 - c_2| > \frac{2}{r^{1/2}},$$
then 
$$E_v ~\geq~ \frac{1}{16} \log r - \frac{3}{16}\log 2.$$
\end{theorem}

We also prove an estimate on $\lambda_c$ from above, at points near the $v$-adic Julia set of $f_c$, that will be needed for the proof of Theorem \ref{boundingN}.

\subsection{Structure of the Julia set} \label{structure 2}
As in the previous section, we work with the dynamics of $f_c$ on the Berkovich affine line ${\bf A}^{1,an}_v$, associated to the complete and algebraically closed field $\C_v$, and we denote by $\zeta_{x,r}$ the Type II point corresponding to the disk of radius $r \in \Q_{>0}$ about $x$.  

And as before, for $|c| \leq 1$, the map $f_c$ has good reduction, and $J_c = \zeta_{0,1}$ is the Gauss point.  For $|c| > 1$ and for any point $z$ with absolute value $|z|> |c|^{1/2}$, we have $|f^n(z)| = |z|^{2^n}$ for all $n\geq 1$, so that 
\begin{equation} \label{lambda value}
	\lambda_c(z) = \log|z|.
\end{equation}
It is also the case that 
\begin{equation} \label{upper bound on lambda}
	\lambda_c(z) \leq \frac12 \log|c|
\end{equation}
for all $|z| \leq |c|^{1/2}$.  

But unlike the setting of the previous section, the geometry of the Julia set and the dynamics on the associated tree is not constant for all $|c| > 1$.  First, for $1 < |c| \leq 4$, the map $f_c$ has {\em potential good reduction}, so its Julia set is a single Type II point.  For all $|c| > 4$, the Julia set will be a Cantor set of Type I points.  As in the previous section, the Julia set and all iterated preimages of $z=0$ are contained in $\{z \in \C_v:  |z| = |c|^{1/2}\}$, for all $|c| > 4$.  We refer to \cite{Benedetto:Briend:Perdry} for basic information about the Julia set.  

It is important to observe that, for any point $z$ with $|z| = |c|^{1/2}$, we have 
	$$|z - (-z)| = |2z| = |z|/2 = |c|^{1/2}/2,$$
a fact we will use repeatedly in our computations.  Distances between points scale as follows:

\begin{lemma} \label{scale factor}
Suppose $|c|>4$ and $z$ is in the Julia set of $f_c$.  For any $|y| > |c|^{1/2}/2$, we have
	$$|f_c(z+y) - f_c(z)| = |y|^2.$$
For $|y| < |c|^{1/2}/2$, we have
	$$|f_c(z+y) - f_c(z)| = |y| |c|^{1/2}/2.$$
\end{lemma}

\begin{proof}
Computing the image of $z+y$, we have
	$$f_c(z+y) = (z+y)^2 + c = (z^2+c) + (y^2 + 2y z).$$
Because $z$ lies in the Julia set, we know that $|z| = |c|^{1/2}$, and the result follows.  
\end{proof}

Note that $|c|^{1/2}/2  > 2$ if and only if $|c| > 16$.  We choose $b$ so that $f_c(\pm b) = 0$.  In the case when $|c|>16$, we let $\beta$ and $\beta'$ be further preimages of 0, so that 
\begin{equation} \label{beta definition}
	f_c(\beta) = b, \quad  f_c(\beta') = -b, \mbox{ with } |\beta - \beta'| = 1 \mbox{ and } |\beta - b| = 2.
\end{equation}
Indeed, this is possible because 
	$$(\beta-\beta')(\beta+\beta') = \beta^2 - \beta'^2 = 2b$$
has absolute value $|c|^{1/2}/2$, and so does $|\beta' - (-\beta')|$, so we can assume that $|\beta + \beta'| = |c|^{1/2}/2$ and $|\beta - \beta'| = 1$.  Moreover, as $x = \beta-b$ is a root of the equation $x^2 + 2bx - b = 0$, a Newton polygon argument shows that $|\beta - b|$ can be chosen to be 2, for $|c| > 16$.  Similarly, we choose further preimages $\gamma$ and $\gamma'$ of $0$ so that 
\begin{equation} \label{gamma definition}
	f_c(\gamma) = \beta, \quad  f_c(\gamma') = \beta', \mbox{ with } |\gamma - \gamma'| = 2/|c|^{1/2} \mbox{ and } |\gamma - \beta| = 4/|c|^{1/2}.
\end{equation}
The structure of the Julia set is shown in Figure \ref{nonarch Julia 2} for $|c| > 16$, and it will be useful to refer to the figure while reading the proof of Theorem \ref{p2}.

\begin{figure} [h]
\includegraphics[width=5in]{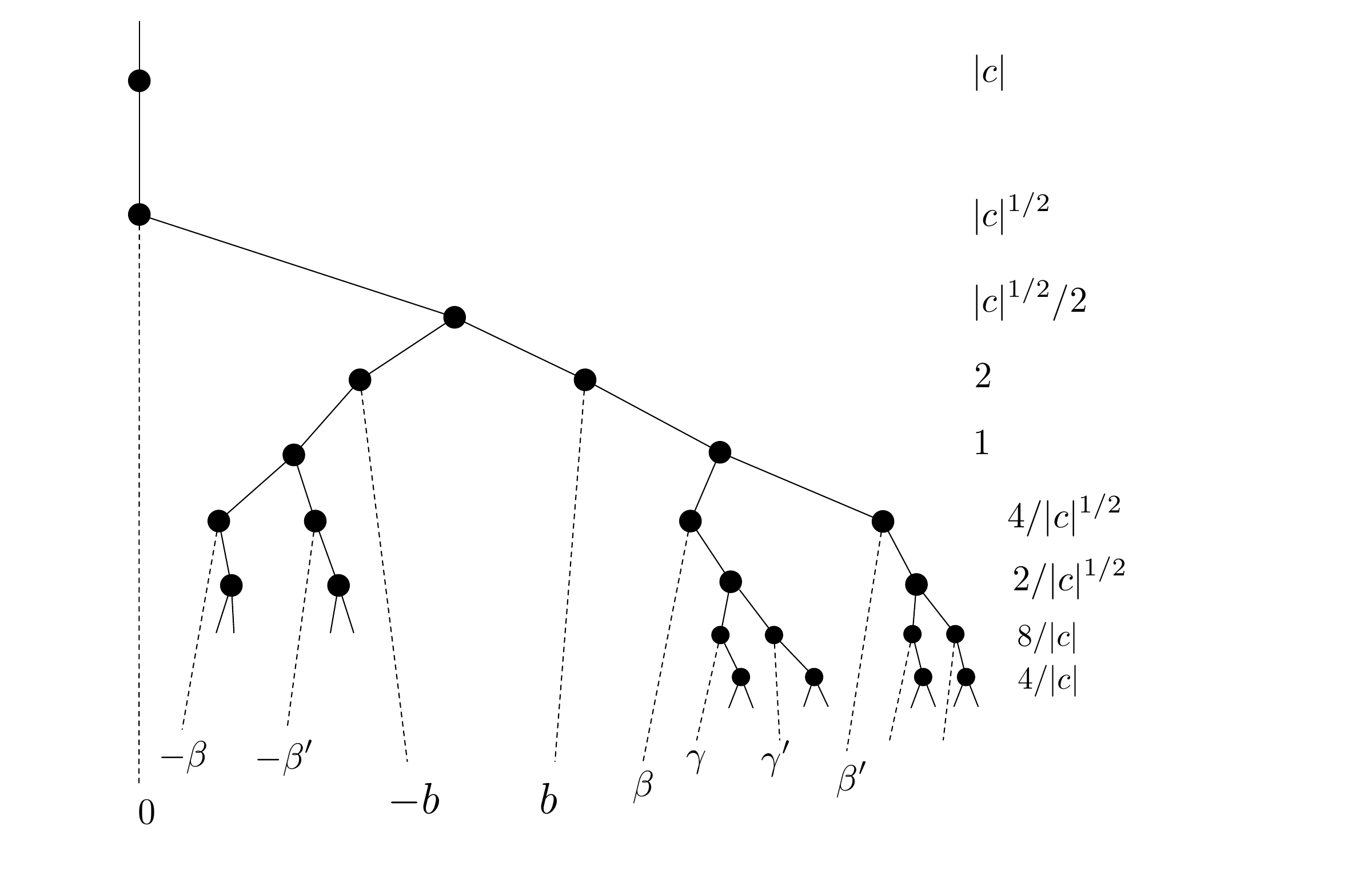}
\caption{ \small The tree structure of the non-archimedean Julia set and some iterated preimages of $0$, for $|c|_v >16$ and $v\mid 2$, vertically ordered by $| \cdot|_v$ as noted on the right.  The solid edges lie in the convex hull of the Julia set and $\infty$.}
\label{nonarch Julia 2}
\end{figure}

\subsection{Proof of Theorem \ref{p2}.} 
If $|c_i| \leq 1$ for at least one $i$, then 
$$E_v =  \frac12 \max \{ \log^+|c_1|, \log^+|c_2| \} = \frac12 \log^+|c_1-c_2|,$$
proving the theorem in this case.  If $1 < |c_1| < |c_2|$, then all points $z \in J_{c_2}$ satisfy $|z| = |c_2|^{1/2} > |c_1|^{1/2}$, so that $\lambda_{c_1}(z) = \log|z| = \frac12 \log|c_2|$ from \eqref{lambda value}, giving 
$$E_v =  \frac12 \max \{ \log^+|c_1|, \log^+|c_2| \} = \frac12 \log^+|c_2-c_1|.$$
Similarly for $1 < |c_2| < |c_1|$, and this completes the proof of the theorem for $|c_1| \not= |c_2|$. 

Note that whenever $1 < |c_2| = |c_1|$, we have $\lambda_{c_1}(z) \leq \frac12 \log|c_1|$ for all $z\in J_{c_2}$, from \eqref{upper bound on lambda}.  It follows that 
	$$E_v \leq \frac12 \max \{ \log^+|c_1|, \log^+|c_2| \},$$ 
proving the upper bound on $E_v$ in all cases.

For $1 < |c_1| = |c_2| \leq 16$, we have $|c_1- c_2| \leq 16$, so that $\frac{1}{16} \log|c_1-c_2| \leq \frac14 \log 2$.  This completes the proof of the first statement of the theorem in this case as well.  

For the remainder of the proof, we assume that
	$$r := |c_1| = |c_2|  > 16.$$
Exactly as in the proofs of Theorems \ref{archimedeanbounds} and \ref{not 2 bound}, we break the proof into cases, according to how close the two parameters are.  As in \S\ref{structure 2}, we let $\pm b_i$ denote the preimages of 0 by $f_{c_i}$.  

\medskip
{\bf Case 1.}  Assume that the preimages $b_1$ and $b_2$ are chosen so that 
	$$s := |b_1 - b_2| \leq |b_1 + b_2|$$
and suppose that they satisfy
	$$r^{1/2}/2 < s  \leq r^{1/2}.$$
Since $|b_2 - (-b_2)| = |b_2|/2= r^{1/2}/2$, so $|b_1 + b_2| = |b_1 - b_2 + 2b_2| = s.$  Then as
\begin{equation} \label{b product 2}
	(b_1 - b_2)(b_1+b_2) = b_1^2 - b_2^2 =  c_2 - c_1,
\end{equation}
it follows that 
	$$|c_1 - c_2| = s^2.$$
For all $z \in J_{c_2}$, we have
	$$\inf_{z_1 \in J_{c_1}} |z - z_1| = s > r^{1/2}/2,$$
so that  
	$$|f_{c_1}^n(z)| = s^{2^n}$$
for all $n\geq 1$, from Lemma \ref{scale factor}.  This gives
	$$\lambda_{c_1}(z) = \log s$$
for all $z \in J_{c_2}$. Therefore,
	$$E_v = \log s = \frac12 \log|c_1-c_2| > \frac12 \log r - \log 2 \geq \frac14 \log r$$
for every $r> 16$.

\medskip
{\bf Case 2.}  Assume that the preimages $b_1$ and $b_2$ are chosen so that 
	$$s := |b_1 - b_2| \leq |b_1 + b_2|$$
and satisfy
	$$1 < s  \leq r^{1/2}/2.$$
Then $|b_1 + b_2| = r^{1/2}/2$, so that
	$$|c_1 - c_2| = s r^{1/2}/2$$
from \eqref{b product 2}.  Choosing $\beta_i$ and $\beta'_i$, $i = 1,2$, as in \eqref{beta definition}, we have 
\begin{equation} \label{beta product 2}
(\beta_1 - \beta_2)(\beta_1+\beta_2) = \beta_1^2 - \beta_2^2 = b_1 - c_1 - (b_2 - c_2) = (b_1 - b_2) + (c_2 - c_1),
\end{equation}
and
\begin{equation} \label{beta' product 2}
(\beta'_1 - \beta'_2)(\beta'_1+\beta'_2) = (\beta'_1)^2 - (\beta'_2)^2 = (b_2 - b_1) + (c_2 - c_1).
\end{equation}
Noting that the expressions in \eqref{beta product 2} and \eqref{beta' product 2} have absolute value $s r^{1/2}/2$, we find that 
	$$|\beta_1 - \beta_2| = |\beta_1' - \beta_2'| = s.$$
It follows that 
	$$\inf_{z_1 \in J_{c_1}} |z - z_1| = s$$
for all $z \in J_{c_2}$.  Therefore
	$$\inf_{z_1 \in J_{c_1}} |f_{c_1}(z) - z_1| = sr^{1/2}/2$$
for all $z \in J_{c_2}$, so that
	$$|f_{c_1}^n(z)| = (sr^{1/2}/2)^{2^{n-1}}$$
for all $n\geq 2$ and $z \in J_{c_2}$.  This implies that
	$$\lambda_{c_1}(z) = \frac12 \log(sr^{1/2}/2)$$
for all $z \in J_{c_2}$, and 
	$$E_v = \frac12 \log(sr^{1/2}/2) = \frac12 \log|c_1-c_2|  > \frac14 \log r - \frac12 \log 2 > \frac18 \log r$$
for all $r> 16$.

\medskip
{\bf Case 3.}  Assume that the preimages $b_1$ and $b_2$ are chosen so that 
	$$1 = |b_1 - b_2| < |b_1 + b_2| = r^{1/2}/2.$$
Then 
	$$|c_1 - c_2| = r^{1/2}/2$$
from \eqref{b product 2}.  It follows that 
	$$|\beta_1 - \beta_2| = |\beta_1' - \beta_2'| = 1,$$
from \eqref{beta product 2} and \eqref{beta' product 2}.  We also have
\begin{equation} \label{mixed beta}
(\beta_1 - \beta_2')(\beta_1+\beta_2') = \beta_1^2 - (\beta_2')^2 = b_1 - c_1 - (-b_2 - c_2) = (b_1 + b_2) + (c_2 - c_1).
\end{equation}
The right-hand-side is the sum of two terms with the same absolute value and may lead to cancellation, so it could happen that $D(\beta_1, 1) = D(\beta_2', 1)$.  On the other hand, we also have 
\begin{equation} \label{mixed beta 2}
(\beta_1' - \beta_2)(\beta_1'+\beta_2) = (\beta_1')^2 - \beta_2^2 = -b_1 - c_1 - (b_2 - c_2) = -(b_1 + b_2) + (c_2 - c_1),
\end{equation}
and $|(b_1+b_2) - (-(b_1+b_2))| = |2||b_1+b_2| = r^{1/2}/4$.  In other words, the cancellation on the right-hand-sides of \eqref{mixed beta} and \eqref{mixed beta 2} cannot bring us smaller than $r^{1/2}/4$ in both equations.  Consequently, we have 
	$$|\beta_1 - \beta_2'| \mbox{ or } |\beta_1' - \beta_2|  \geq  (r^{1/2}/4) / (r^{1/2}/2) = \frac12.$$
Consequently, at least half of the Julia set $J_{c_2}$ (with respect to the measure $\mu_{c_2}$) must be at distance at least $1/2$ from the Julia set $J_{c_1}$.  Note that $r> 16$ implies that $1/2 > 2/r^{1/2}$.  So, for half of the points $z \in J_{c_2}$, we have
	$$\inf_{z_1 \in J_{c_1}} |f_{c_1}^2(z) - z_1| \geq  \frac12 \left(\frac{r^{1/2}}{2}\right)^2 = \frac{r}{8},$$
and thus
	$$\lambda_{c_1}(z) \geq \frac14 \log(r/8) = \frac14 (\log r - \log 8)$$
for these $z$ values.  We conclude that 
\begin{eqnarray*}
E_v &\geq& \frac18 \log(r/8) ~= ~ \frac14 \log|c_1-c_2| - \frac18 \log 2  \\
	&=& \frac18 \log r - \frac38 \log 2 \\
	&\geq& \frac{1}{16} \log r - \frac18 \log 2
\end{eqnarray*}
for all $r> 16$.
	
\medskip
{\bf Case 4.}  Assume that the preimages $b_1$ and $b_2$ are chosen so that 
	$$s := |b_1 - b_2| < |b_1 + b_2| = r^{1/2}/2$$
and satisfy
	$$2/r^{1/2} < s  < 1.$$
Then \eqref{b product 2} implies that
	$$1 < |c_1 - c_2| = s r^{1/2}/2 < r^{1/2}/2.$$
We have 
	$$|\beta_1 - \beta_2| = |\beta_1' - \beta_2'| = s$$
from \eqref{beta product 2} and \eqref{beta' product 2}.  We now choose $\gamma_i$ and $\gamma_i'$, $i = 1,2$, as in \eqref{gamma definition}, and these satisfy
\begin{equation} \label{gamma product}
(\gamma_1 - \gamma_2)(\gamma_1 + \gamma_2) = \gamma_1^2 - \gamma_2^2 = (\beta_1-\beta_2) + (c_2- c_1).
\end{equation}
Similarly for $\gamma_i'$.  Consequently, 
	$$|\gamma_1 - \gamma_2| = |\gamma'_1 - \gamma_2'| = s.$$
It follows that all points $z\in J_{c_2}$ are distance $s$ from $J_{c_1}$, so that 
	$$\inf_{z_1 \in J_{c_1}} |f_{c_1}^2(z) - z_1| = (r^{1/2}/2)^2 s = rs/4$$
and 
	$$|f^n_{c_1}(z)| = (rs/4)^{2^{n-2}} \mbox{ for all } n\geq 3$$
for all $z \in J_{c_2}$.  Therefore, 
	$$\lambda_{c_1}(z) = \frac14 \log(rs/4)$$
for all $z \in J_{c_2}$, so that
	$$E_v = \frac14 \log(rs/4) = \frac12 \log|c_1 - c_2| - \frac14 \log s > \frac12 \log|c_1-c_2|$$
and 
	$$E_v \geq \frac14 \log(r^{1/2}/2) \geq \frac{1}{16} \log r $$
for all $r > 16$.

\medskip
{\bf Case 5.}  Assume that the preimages $b_1$ and $b_2$ satisfy
	$$ 2/r^{1/2} = |b_1 - b_2| < |b_1 + b_2| = r^{1/2}/2.$$
Then
	$$|c_1 - c_2| = 1$$
from \eqref{b product 2}.  Equations \eqref{beta product 2} and \eqref{beta' product 2} imply that 
\begin{equation} \label{beta close}
	|\beta_1 - \beta_2| = |\beta_1'- \beta_2'| = 2/r^{1/2},
\end{equation}
and \eqref{mixed beta} and \eqref{mixed beta 2} imply that
	$$|\beta_1 - \beta_2'| = |\beta_1' - \beta_2| = 1.$$ 
To determine how the Julia sets might overlap, we examine third preimages of 0.  From \eqref{gamma product}, we know that 
	$$|\gamma_1 - \gamma_2| = |\gamma_1' - \gamma_2'| = 2/r^{1/2}.$$
But 
\begin{equation} \label{mixed gamma}
	(\gamma_1 - \gamma_2')(\gamma_1 + \gamma_2') = \gamma_1^2 - (\gamma_2')^2 = (\beta_1-\beta_2') + (c_2- c_1)
\end{equation} 
and both terms on the right-hand-size have absolute value 1.  So it can happen that $D(\gamma_1, 2/r^{1/2}) = D(\gamma_2', 2/r^{1/2})$.  Similarly for $\gamma_1'$ with $\gamma_2$.  But both pairs cannot be too close, because
	$$(\beta_1 - \beta_2') - (\beta_1' - \beta_2) = (\beta_1 - \beta_1') + (\beta_2 - \beta_2') = (\beta_1 - \beta_1') + (\beta_1 - \beta_1') + \eps$$
for some $|\eps| \leq 2/r^{1/2}$, from \eqref{beta close}.  It follows that 
	$$|(\beta_1 - \beta_2') - (\beta_1' - \beta_2)| = |2(\beta_1 - \beta_1')| = \frac12$$
so that 
	$$|\gamma_1 - \gamma_2'| \mbox{ or } |\gamma_1' - \gamma_2|  \geq  (1/2) / (r^{1/2}/2) = \frac{1}{r^{1/2}}.$$
The same estimates will hold for the third preimages of 0 near $\beta_i'$, as well as those near $- \beta_i$ and $-\beta'_i$.  Consequently, at least half of the Julia set $J_{c_2}$ (with respect to $\mu_{c_2}$) must be at distance at least $1/r^{1/2}$ from the Julia set $J_{c_1}$.  Note that $r> 16$ implies that $1/r^{1/2} > 4/r$.  So, for these points $z \in J_{c_2}$, we have
	$$\inf_{z_1 \in J_{c_1}} |f_{c_1}^3(z) - z_1| \geq \left(\frac{r^{1/2}}{2}\right)^3 \frac{1}{r^{1/2}} = \frac{r}{8},$$
and thus
	$$\lambda_{c_1}(z) \geq \frac18 \log(r/8) = \frac18 \log r - \frac18 \log 8$$
for these $z$ values.  We conclude that 
	$$E_v \geq \frac{1}{16} \log(r/8) = \frac{1}{16} \log r - \frac{3}{16} \log 2.$$

\medskip
{\bf Case 6.}  Assume that the preimages $b_1$ and $b_2$ are chosen so that 
	$$s := |b_1 - b_2| \leq |b_1 + b_2|$$
and satisfy
	$$4/r < s < 2/r^{1/2}.$$
Then 
	$$2/r^{1/2} < |c_1 - c_2| = sr^{1/2}/2 < 1$$
from \eqref{b product 2}.  We also compute 
	$$|\gamma_1 - \gamma_2| = |\beta_1 - \beta_2| = s$$
from \eqref{beta product 2} and \eqref{gamma product}.  But, for disks centered at the 3rd preimages of 0 to contain the Julia set, we need to take radius $8/r$, which may be larger than $s$.  So we pass to 4th preimages $\delta_i$ of 0, so that $f_{c_i}(\delta_i) = \gamma_i$; observe that we can choose these so that 
	$$|\delta_1 - \delta_2| = s,$$
because $(\delta_1 - \delta_2)(\delta_1 + \delta_2) = \delta_1^2 - \delta_2^2 = (\gamma_1 - \gamma_2) + (c_2 - c_1)$.  This is enough to conclude that 
	$$\inf_{z_1 \in J_{c_1}} |z - z_1| = s$$
for all $z \in J_{c_2}$.  Therefore, 
	$$\inf_{z_1 \in J_{c_1}} |f^3_{c_1}(z) - z_1| = s(r^{1/2}/2)^3 > r^{1/2}/2$$
for all $z \in J_{c_2}$, so that
	$$\lambda_{c_1}(z) = \frac18 \log(s r^{3/2}/8)$$
for all $z \in J_{c_2}$, and  
	$$E_v = \frac18 \log(s r^{3/2}/8) \geq \frac18 \log(r^{1/2}/2) = \frac{1}{16} \log r - \frac18 \log 2.$$

Finally, note that if $|b_1 - b_2| \leq 4/r$, then $|c_1 - c_2| \leq 2/r^{1/2}$, so Case 6 completes the proof of the theorem.

\subsection{An upper bound on the local height near the Julia set}
We will use the following proposition in the proof of Theorem \ref{boundingN}.  This is an analog of the estimates provided in Propositions \ref{archimedean epsilon} and  \ref{escapebound pnot2}.

\begin{prop} \label{escapebound p2} 
Suppose $v$ is a non-archimedean place of $K$ dividing 2.  
For any $0 < r < 1/4$, we have 
	$$\lambda_{c}(z) \leq r \log\max\{|c|, 16\}$$
for all $z$ within distance 
	$$\frac{1}{\max\{|c|, 16\}^{\log(1/r)}}$$
of the filled Julia set within $\C_v$.  
\end{prop}

\begin{proof} 
First assume that $|c| > 4$.  Recall that all points $x$ of the Julia set $J_c$ (which agrees with the filled Julia set in this setting) satisfy $|x| = |c|^{1/2}$.  From Lemma \ref{scale factor}, we know that for all $x \in J_c$ and all $z = x+y$ with $|y| < |c|^{1/2}/2$, we have 
$$|f_c(z) - f_c(x)|_v = |2x y + y^2 |_v = |y| |c|^{1/2}/2.$$
Recall also that $\lambda_c(z) = \log |z|$ for all $|z| > |c|^{1/2}$ and $\lambda_c(z) \leq \frac12 \log |c|$ for all $|z| \leq |c|^{1/2}$.  

In particular, for $|c| > 4$ and for any $n\geq 2$, a point $z$ within distance 
	$$\frac{|c|}{4} \left(\frac{2}{|c|^{1/2}}\right)^n \geq \frac{1}{|c|^{(n/2)-1}}$$ 
will satisfy 
	$$\lambda_c(z) = 2^{-n} \lambda_c(f^n(z)) \leq \frac{1}{2^n} \log |c|.$$
Fix $r \in (0, 1/4)$ and choose $n\geq 3$ so that $\frac{1}{2^n} \leq r < \frac{1}{2^{n-1}}$.  Note that 
	$$\log(1/r) > (n-1) \log 2 > \frac{n}{2} - 1$$
for all $n\geq 3$.  So if $z$ is within distance $1/|c|^{\log(1/r)}$ of the Julia set, then 
	$$\lambda_c(z) \leq \frac{1}{2^n} \log|c| \leq r \log |c| \leq r \log \max\{|c|, 16\}.$$

Now assume $|c|_v \leq 4$.  Then $f_c$ has potential good reduction with $J_c = \zeta_{x,1}$, where $x$ is any element of the filled Julia set.  Consequently, all points $z$ within distance 1 of the filled Julia set are in the filled Julia set and thus satisfy $\lambda_c(z) =0$.  
\end{proof}

\bigskip
\section{Bounds on the energy pairing} \label{pairingboundsection}

In this section, we use the estimates of the previous sections to prove a weak version of Theorem \ref{pairingbound}, and we use it to deduce Theorem \ref{uniformbound}.  We let $h(x)$ denote the logarithmic Weil height of $x \in \Qbar$ and $h(x_1, x_2)$ the Weil height on $\A^2(\Qbar)$.  

\begin{theorem} \label{weak pairingbound}
We have 
	$$\frac{1}{16} h(c_1-c_2)  - \frac23 ~\leq~ \<f_{c_1}, f_{c_2} \> ~\leq~  \frac12 h(c_1, c_2) + \frac73$$
for all $c_1\not=c_2$ in $\Qbar$.
\end{theorem}

\subsection{Proof of Theorem \ref{weak pairingbound}} \label{weaklowerbound}
Fix $c_1 \not= c_2$ in $\Qbar$, and let $K$ be any number field containing them.  Summing over all places of $K$, we have by Theorem \ref{archimedeanbounds}, Theorem \ref{not 2 bound}, and Theorem \ref{p2} that 
$$\frac{1}{16} \sum_{v \in M_K} \frac{[K_v: \Q_v]}{[K:\mathbb{Q}]} \log^+|c_1 - c_2|   - \frac{1}{16}\log 2000
- \frac{1}{4} \log 2 \quad  \leq  \quad \< f_{c_1}, f_{c_2} \>  $$
 $$\qquad\qquad \leq \frac{1}{2} \sum_{v \in M_K} \frac{[K_v: \Q_v]}{[K:\mathbb{Q}]} \log^+ \max \{ |c_1|_v, |c_2|_v \} + \log 8,$$
where the added constants come from the archimedean places (Remark \ref{arch constants}) and the prime 2.  This completes the proof of the theorem.

\subsection{Proof of Theorem \ref{uniformbound}} \label{proof of delta}
We will assume towards contradiction that there is a sequence of triples $c_{1,n} \not= c_{2,n}\in  \Qbar$ and $\eps_n>0$ such that 
   $$\<f_{c_{1,n}}, f_{c_{2,n}}\><\eps_n,$$
where $\eps_n\to 0$ as $n$ tends to infinity.  Let $K_n$ be a number field containing $c_{1,n}$ and $c_{2,n}$.  We will show that this forces the pairing at a (proportionally) large number of archimedean places of $K_n$ to be close to 0; as a consequence we will deduce that the height $h(c_{1,n}- c_{2,n})$ must get large.  This in turn would contradict Theorem \ref{weak pairingbound}. 

Let $M_n^\infty$ denote the set of all archimedean places of $K_n$.  For each $v \in M_n^\infty$, we let 
	$$E_v(c_{1,n}, c_{2,n}) = \int \lambda_{c_1,v} \, d\mu_{c_2, v}$$
denote the local contribution to the energy pairing.  We let $S_n\subset \M^\infty_{n}$ be the set of archimedean places with
  $$E_v(c_{1,n}, c_{2,n})< 2\eps_n.$$
Since $\sum_{v\in \M^\infty_{n}} [K_{n,v}:\Q_v]=[K_n: \Q]$ and $\<f_{c_{1,n}}, f_{c_{2,n}}\> < \eps_n$, we see that $\sum_{v\in \M^\infty_{n}\setminus S_n} [K_{n,v}:\Q_v] <   [K_n: \Q]/2$.  Therefore, 
  $$\sum_{v\in S_n} [K_{n,v}:\Q_v] \geq  \frac{[K_n: \Q]}{2}.$$
Take $L=1000$ as in Remark \ref{arch constants}, and choose any $M > L$.  

Recall that, for a fixed archimedean place $v|\infty$, we have $\mu_{c_1} = \mu_{c_2}$ if and only if $c_1=c_2$ from \eqref{known equivalence}, so that $E_\infty(c_1, c_2) > 0$ for all $c_1\not= c_2 \in \C = \C$.  Moreover, $E_\infty$ is continuous as a function of $(c_1, c_2)$ because of the continuity of $\lambda_c(z)$ in $c$ and $z$ and the (weak) continuity of the measures $\mu_c = \frac{1}{2\pi} \Delta \lambda_c$.  Therefore, for any $\delta>0$, $E_\infty(c_1,c_2)$ obtains a positive minimum on the compact set where $|c_1-c_2| \geq  \delta$ and $|c_1|, |c_2| \leq M$ for $c_1, c_2\in \C$.  It follows that there is a sequence $\delta_n\to 0^+$ as $n\to \infty$ such that 
  	$$E_\infty(c_{1}, c_{2})\geq 2\eps_n $$
for all $|c_1-c_2| \geq  \delta_n$ and $|c_1|, |c_2| \leq M$ for $c_1, c_2\in \C$.  Furthermore, if one of the $c_i$, say $c_1$, has absolute value bigger than $M$ and if $|c_1-c_2| >3/|c_1|^{1/2}$, then
  	$$E_\infty(c_1, c_2)\geq \frac{1}{64} \log |c_1| \geq \frac{1}{64}\log M$$
from Theorem \ref{archimedeanbounds}. 

For all $n$ sufficiently large, we have $2\eps_n< \frac{1}{64}\log M$, and so for any $v\in S_n$, as $E_v(c_{1,n}, c_{2,n})< 2\eps_n$, we must have
  $$|c_{1,n} -c_{2,n}|_v\leq \max\left\{ \delta_n, \frac{3}{M^{1/2}}\right\}.$$
Hence for any $n$ large enough that $2\eps_n< \frac{1}{64}\log M$ and $\delta_n<3/M^{1/2}$, we conclude that 
  $$|c_{1,n} -c_{2,n}|_v\leq 3/M^{1/2} <1$$
for all $v\in S_n$.  Consequently, 
\begin{eqnarray*}
h(c_{1, n}-c_{2,n})
&\geq& \sum_{v\in \M_{K_n}\backslash S_n}\frac{ [K_{n,v}:\Q_v]}{[K_n:\Q]} \log^+ |c_{1,n}-c_{2,n}|_v\\
&\geq& \sum_{v\in \M_{K_n}\backslash S_n}\frac{ [K_{n,v}:\Q_v]}{[K_n:\Q]} \log |c_{1,n}-c_{2,n}|_v\\
  &=&  \sum_{v\in  S_n} \frac{ [K_{n,v}:\Q_v]}{[K_n:\Q]}  \log\frac{1}{ |c_{1,n}-c_{2,n}|_v}\\
  &\geq&  \left(\sum_{v\in  S_n}\frac{ [K_{n,v}:\Q_v]}{[K_n:\Q]}  \right)\log\frac{M^{1/2}}{3}  \; \geq \; \frac{1}{2}\log\frac{M^{1/2}}{3}.
\end{eqnarray*}  
We thus have by Theorem \ref{weak pairingbound} that
$$\frac{1}{32} \log \frac{M^{1/2}}{3} - \frac23 \leq \<f_{c_{1,n}}, f_{c_{2,n}}\> < \epsilon_n,$$
for any choice of $M > L$ and all sufficiently large $n$.  But this is a clearly a contradiction for $M$ and $n$ large enough.
\qed

\bigskip
\section{Strong lower bound on the energy pairing} \label{stronglowerbound}

Throughout this section, we assume that $c_1$ and $c_2$ are distinct elements of $\Qbar$.  We prove Theorem \ref{pairingbound}, which gives bounds on the energy pairing $\<f_{c_1}, f_{c_2}\>$ in terms of the heights of the parameters.

The upper bound in Theorem \ref{pairingbound} is easy and was stated as part of Theorem \ref{weak pairingbound}.  The lower bound is a balancing act between ``helpful" primes and the other primes of a given number field $K$ containing the pair $c_1$ and $c_2$.  A place $v$ of $K$ will be helpful if at least one absolute value $|c_i|_v$ is large and the two parameters are not too close in the $v$-adic distance.  In this good setting, we can apply the stronger lower bounds on the local energy pairing, as in the second statement of Theorem \ref{archimedeanbounds}.  By showing that a significant proportion of primes are helpful, we obtain the lower bound of Theorem \ref{pairingbound}.

\subsection{An auxiliary height} \label{h_L}
Fix some constant $L>1$ and consider the following function $h_L$ on $\A^2(\Qbar)$.  For $c_1, c_2$ in a number field $K$, we put 
	$$r_v = [K_v:\Q_v]/[K:\Q],$$ 
and set
	$$\ell_v = \left\{ \begin{array}{ll}  
		\log\max\{|c_1|_v, |c_2|_v, L\} & \mbox{for } v \mbox{ archimedean} \\
		\log\max\{|c_1|_v, |c_2|_v, 16\} & \mbox{for } v | 2 \\
		\log\max\{|c_1|_v, |c_2|_v, 1\} & \mbox{otherwise} 
		\end{array} \right.$$
and define
	$$h_L(c_1, c_2) := \sum_{v\in M_K} r_v \ell_v.$$
Note that
	$$h(c_1, c_2) \leq h_L(c_1, c_2) \leq h(c_1, c_2) + \log L + \log 16,$$
where $h(c_1, c_2)$ is the usual logarithmic Weil height on $\A^2(\Qbar)$.

\subsection{Helpful places}   \label{helpful}	
With $L>1$ fixed, and elements $c_1$ and $c_2$ in the number field $K$, we say that the quantity $\ell_v$ is large if 
	$$\ell_v  > \left\{ \begin{array}{ll}  
		\log L & \mbox{for } v \mbox{ archimedean} \\
		\log 16 & \mbox{for } v \mid 2 \\
		0 & \mbox{otherwise.} \end{array} \right.$$
We define $\Mgood$ to be the subset of $M_K$ for which $\ell_v$ is large and
	$$|c_1 - c_2|_v >  \kappa_v e^{-\ell_v/2},$$
where 
	$$\kappa_v = \left\{ \begin{array}{ll}
		3 & \mbox{for } v \mbox{ archimedean} \\
		2 & \mbox{for } v \mid 2 \\
		1 & \mbox{otherwise} \end{array} \right.$$
and we call these places ``helpful".   We define $\Mbad$ to be the subset of $M_K$ for which $\ell_v$ is large and 
	$$|c_1 - c_2| \leq  \kappa_v e^{-\ell_v/2}$$
and call these places ``close".  We will say that a place $v$ is in $\Mbounded$ if $\ell_v$ fails to be large.

The helpful places constitute a significant portion of the contribution to the height:

\begin{lemma}  \label{most are helpful}
For any $c_1, c_2 \in \Qbar$ and any $L\geq 1$, we have 
	$$\sum_{v \in M_K\setminus \Mbad} r_v \ell_v \; \geq \; \frac{1}{3} h_L(c_1, c_2) - \frac{2}{3} \log 6.$$
and
	$$\sum_{v \in \Mgood} r_v \ell_v \; \geq\;  \frac{1}{3}h_L(c_1, c_2) - \log(16 \cdot 6^{2/3}\cdot L)$$
for any $c_1, c_2 \in \Qbar$ and any $L\geq 1$.
\end{lemma}

\begin{proof}
We use the product formula on $c_1 - c_2$, so that 
	$$1 = \prod_v |c_1 - c_2|_v^{r_v}.$$
At the close places, we know that $|c_1 - c_2|$ is bounded from above by $\kappa_v e^{-\ell_v/2}$.  At all other places, we have $|c_1 - c_2|_v \leq e^{\ell_v}$ if non-archimedean, and $|c_1 - c_2|_v \leq 2 e^{\ell_v} \leq \kappa_v e^{\ell_v}$ if archimedean.  Therefore, we have 
\begin{eqnarray*}
1 &\leq& \prod_{v \in \Mbad} (\kappa_v e^{-\ell_v/2})^{r_v}  \prod_{v \in M_\infty \setminus \Mbad} (\kappa_v e^{\ell_v})^{r_v}  \prod_{v \in M_K \setminus (M_\infty \cup \Mbad)} (e^{\ell_v})^{r_v} \\
&\leq &  6 \prod_{v \in \Mbad} (e^{-\ell_v/2})^{r_v} \prod_{v \in M_K \setminus \Mbad} (e^{\ell_v})^{r_v}.
\end{eqnarray*}
Taking logarithms gives 
\begin{equation} \label{bad on left}
	\frac12 \sum_{v \in \Mbad} r_v \ell_v  \; \leq  \sum_{v \in M_K\setminus \Mbad} r_v \ell_v + \log 6.
\end{equation}
Adding $\frac12 \sum_{v \in M_K\setminus \Mbad} r_v \ell_v$ to both sides yields 
	$$\frac12 h_L(c_1, c_2)  \; \leq \frac{3}{2} \sum_{v \in M_K\setminus \Mbad} r_v \ell_v  + \log 6,$$
proving the first statement of the lemma.

Expanding the right-hand-side of \eqref{bad on left}, we see that 
$$\frac12 \sum_{v \in \Mbad} r_v \ell_v \leq \sum_{v \in \Mgood} r_v \ell_v + \sum_{v \in \Mbounded} r_v \ell_v + \log 6$$
so that 
$$\sum_{v \in \Mgood} r_v \ell_v \geq   \frac12 \sum_{v \in \Mbad} r_v \ell_v - \sum_{v \in \Mbounded} r_v \ell_v - \log 6.$$
Adding $\frac12 \sum_{v \in \Mgood} r_v \ell_v$ to both sides, we obtain
\begin{eqnarray*}
\frac{3}{2} \sum_{v \in \Mgood} r_v \ell_v 
&\geq& \frac12 h_L(c_1, c_2) - \frac{3}{2} \sum_{v \in \Mbounded} r_v \ell_v - \log 6 \\
&\geq& \frac12 h_L(c_1, c_2) - \frac{3}{2} (\log L + \log 16) - \log 6 \\
&= & \frac12 h_L(c_1, c_2) - \frac{3}{2} \log (16\cdot 6^{2/3} \cdot L),
\end{eqnarray*}
which proves the lemma.
\end{proof}

\subsection{Proof of Theorem \ref{pairingbound}} \label{pairingboundproof} Fix $c_1, c_2$ and choose any number field $K$ containing both.   Fix any $L\geq 1000$ so that Theorem \ref{archimedeanbounds} is satisfied.  Decompose $M_K$ into $\Mgood \cup \Mbad \cup \Mbounded$ as in \S\ref{helpful}.  Note that $\frac{1}{16} \log r - \frac{3}{16} \log 2 \geq \frac{1}{64} \log r$ for any $r\geq 16$.  Then Theorems \ref{not 2 bound},  \ref{p2}, and  \ref{archimedeanbounds} applied in the helpful places combine to say 
\begin{eqnarray} \label{lower bound with L}
\<f_{c_1}, f_{c_2} \> &=& \sum_{v \in M_K} r_v E_v \; \geq \sum_{v \in \Mgood} r_v E_v \nonumber \\
	&\geq& \; \frac{1}{64} \sum_{v \in \Mgood} r_v \log\max\{|c_1|_v, |c_2|_v\}  \nonumber \\
	&=& \; \frac{1}{64}\sum_{v \in \Mgood} r_v \ell_v.
\end{eqnarray}
Combined with Lemma \ref{most are helpful}, this proves that for all $c_1$ and $c_2$ in $\Qbar$, we have 
\begin{eqnarray*}
\<f_{c_1}, f_{c_2} \> 
&\geq&    \frac{1}{3\cdot 64}h_L(c_1, c_2) - \frac{1}{64}\log(16 \cdot 6^{2/3}\cdot L) \\
&\geq&   \frac{1}{192}h(c_1, c_2) - \frac{1}{64}\log(16 \cdot 6^{2/3}\cdot L) . 
\end{eqnarray*}
This proves the lower bound of the theorem with $\alpha_1 = 1/192$ and $C_1 = \frac{1}{64}\log(16 \cdot 6^{2/3}\cdot L)  < 0.17 < \frac{3}{17}$ for $L = 1000$.  The upper bound of the theorem was proved already as Theorem \ref{weak pairingbound} with $\alpha_2 = 1/2$ and $C_2 = 7/3$.

\bigskip
\section{Quantitative equidistribution} \label{QE}

Our goal in this section is to prove Theorem \ref{boundingN}, providing an upper bound on the energy pairing $\<f_{c_1}, f_{c_2}\>$, in terms of the number of common preperiodic points, for $c_1 \not= c_2$ in $\Qbar$, assuming $f_{c_1}$ and $f_{c_2}$ share at least one preperiodic point other than $\infty$.   We build upon the ideas developed in the proof of \cite[Theorem 3]{FRL:equidistribution} and \cite[Theorem 4]{Fili:energy}.

\subsection{Adelic measures and heights on $\P^1(\Qbar)$}
Following Favre and Rivera-Letelier \cite{FRL:equidistribution}, we define the mutual energy of measures $\rho$ and $\rho'$ on $\P^1(\C)$ by
	$$(\rho, \rho') := - \iint_{\C\times\C\setminus \Diag} \log|z-w| \,d\rho(z) \,d\rho'(w),$$
where $\Diag$ is the diagonal, assuming $\log|z-w|$ is in $L^1(\rho\otimes\rho')$.  If the measures have total mass 0 with continuous potentials on $\P^1$, 
we have $(\rho, \rho) \geq 0$ with equality if and only if $\rho = 0$.  Similarly, one defines
\begin{equation} \label{local energy}
(\rho, \rho')_v :=  - \iint_{{\bf A}^1_v \times {\bf A}^1_v \setminus \Diag} \delta_v(z,w) \,d\rho(z) \,d\rho'(w)
\end{equation}
on the Berkovich line over $\C_v$, with respect to a non-archimedean valuation, where $\delta_v(z,w)$ is the logarithm of the Hsia kernel in place of $\log|z-w|_v$.  See \cite[Proposition 4.1]{BRbook} and further information throughout Chapters 4 and 5 of \cite{BRbook}.

Now let $K$ be a number field.  An adelic measure is a collection $\mu = \{\mu_v\}_{v \in M_K}$ of probability measures on the Berkovich ${\bf P}^{1,an}_v$, with continuous potentials at all places $v$ and  for which all but finitely many are trivial (meaning that they are supported at the Gauss point).  For any adelic measure $\mu$, a height function is defined on $\P^1(\Qbar)$ by 
	$$h_\mu(F) := \frac12 \sum_{v \in M_K} \frac{[K_v:\Q_v]}{[K:\Q]} \, ([F] - \mu_v, [F] - \mu_v)_v,$$
where $F$ is any finite, $\Gal(\Kbar/K)$-invariant subset of $\Kbar$, and $[F]$ is the probability measure supported equally on the elements of $F$.  We put
$$h_\mu(\infty) :=  \frac12 \sum_v \frac{[K_v:\Q_v]}{[K:\Q]} (\mu_v, \mu_v)_v.$$

The equidistribution theorems of \cite{FRL:equidistribution, Baker:Rumely:equidistribution, ChambertLoir:equidistribution} state that if $F_n$ is a seqence of $\Gal(\Kbar/K)$-invariant finite sets with $h_\mu(F_n) \to 0$ and $|F_n| \to \infty$ as $n\to \infty$, the discrete probability measures 
	$$\mu_n := \frac{1}{|F_n|} \sum_{x \in F_n} \delta_x$$
converge weakly to the measure $\mu_v$ on ${\bf P}^{1,an}_v$  at each place $v$ of $K$.

There is a pairing between any two such heights, $h_\mu$ and $h_\nu$, associated to adelic measures $\mu$ and $\nu$, as 
\begin{equation} \label{height pairing}
\<h_\mu, h_\nu\> =  \frac{1}{2} \sum_{v \in \M_K}  \frac{[K_v:\Q_v]}{[K:\Q]} ( \mu_v - \nu_v, \mu_v - \nu_v )_v.
\end{equation}
It satisfies $\<h_\mu, h_\nu\> = 0 \iff h_\mu = h_\nu \iff \mu = \nu$.  The energy pairing \eqref{energy pairing} between two quadratic polynomials is a special case, taking the dynamical canonical heights $\hat{h}_{c_1}$ and $\hat{h}_{c_2}$ associated to their adelic equilibrium measures.  

\begin{remark} 
The height $h_\mu$ is defined for an arbitrary adelic measure, but small sequences (meaning the sequences $\{F_n\}$ of Galois-invariant sets with $h_\mu(F_n) \to 0$ and $|F_n| \to \infty$) do not always exist.  
\end{remark}

\subsection{Height pairing as a distance}
Following \cite{Fili:energy}, we consider a distance between two adelic measures $\mu = \{\mu_v\}$ and $\nu = \{\nu_v\}$ on $\P^1$ over a number field $K$, defined by  
   $$d(\mu, \nu) := \<h_\mu, h_\nu\>^{1/2},$$
where $\<h_\mu, h_\nu\>$ was defined in \eqref{height pairing}; see \cite[Theorem 1]{Fili:energy}.

Suppose that $c_1$ and $c_2$ are elements of a number field $K$.  Let $\mu_1:=\{\mu_{c_1, v}\}_{v \in M_K}$ and $\mu_2:=\{\mu_{c_2,v}\}_{v \in M_K}$ be the equilibrium measures of $f_{c_1}$ and $f_{c_2}$, respectively.  Let $F$ be any finite, nonempty, $\Gal(\Kbar/K)$-invariant subset of $\P^1(\Qbar)$.  Let $[F]$ denote the probability measure supported equally on the elements of $F$.  For each place $v$ of $K$, choose a positive real $\eps_v>0$, with $\eps_v = 1$ for all but finitely many $v$.   The collection $\eps := \{\eps_v\}_{v \in M_K}$ will be called an adelic radius.  As in \cite{FRL:equidistribution}, we consider the adelic measure $[F]_\eps$, defined as a regularization of the probability measure $[F]$:  it is supported on the circles of radius $\eps_v$ about each point of $F$.  At a non-archimedean place, this means the Type II or III point associated to the disk of radius $\eps_v$.  The triangle inequality implies that
\begin{equation} \label{triangle distance}
\< f_{c_1}, f_{c_2} \>^{1/2} = d(\mu_1, \mu_2) \leq d(\mu_1, [F]_{\eps}) + d(\mu_2, [F]_{\eps})
\end{equation}
for any choices of $F$ and $\eps$.  

It is worth noting that, if the radius $\eps_v \to 0$ at some place, then the right-hand-side of \eqref{triangle distance} will tend to $\infty$.  This is because the potential of the measure $[F]_\eps$ at $v$ will blow up near the points of $F$.  On the other hand, for $\eps_v$ too large, the measure $[F]_\eps$ is not a good approximation of $[F]$.  Thus, for \eqref{triangle distance} to be useful in our proof of Theorem \ref{preperbound}, we will need to choose $\eps$ well.  This general strategy also appears in the proofs of \cite[Theorem 3]{FRL:equidistribution} and in \cite[Proposition 13]{Fili:energy}.  In our case, the choice of $\eps = \{\eps_v\}_{v \in M_K}$ will be governed by Proposition \ref{archimedean epsilon} and its non-archimedean counterparts, and this leads to Theorem \ref{boundingN}.

\begin{lemma} \label{inequalityallplaces} 
Let $K$ be a number field and fix $c_1 \not= c_2$ in $K$.  We have 
$$ \< f_{c_1}, f_{c_2} \>^{1/2} \leq \sum_{i = 1}^2 \left( \sum_{v \in M_K}  \frac{[K_v:\Q_v]}{[K:\Q]}  \left( - ( \mu_i, [F]_{\eps} )_v +\frac{\log (1/\eps_v)}{2 |F|} \right)\right)^{1/2}$$
for any choice of finite, non-empty, $\Gal(\Kbar/K)$-invariant subset $F$ of $\Qbar$ and any adelic radius $\eps = \{\eps_v\}_{v \in M_K}$.
\end{lemma}

\begin{proof}
We first observe that 
\begin{eqnarray*}
d(\mu_i, [F]_\eps)^2 &=& \frac12 \sum_v \frac{[K_v:\Q_v]}{[K:\Q]} \; (\mu_i - [F]_\eps, \mu_i - [F]_\eps)_v \\
	&=& \sum_v \frac{[K_v:\Q_v]}{[K:\Q]} \left( -(\mu_i, [F]_\eps)_v + \frac12 ([F]_\eps, [F]_\eps)_v \right),
\end{eqnarray*}	
because $(\mu_i, \mu_i)_v = 0$ at every place.  The self-pairing of $[F]_\eps$ can be estimated in terms of the self-pairing of $[F]$ (\cite[Lemma 12]{Fili:energy} and \cite[Lemma 4.11]{FRL:equidistribution}), as
$$([F]_\eps, [F]_\eps )_v \leq ([F], [F])_v + \frac{\log (1/\eps_v)}{|F|}.$$
But observe that 
	$$\sum_v \frac{[K_v:\Q_v]}{[K:\Q]}([F], [F])_v = 0$$
by the product formula on $K$.  So the triangle inequality \eqref{triangle distance} completes the proof of the proposition.
\end{proof}

\subsection{Proof of Theorem \ref{boundingN}}\label{boundingNproof}
Fix any $L \geq 27$, and recall the definition of the auxiliary height $h_L$ on $\A^2(\Qbar)$ from \S\ref{h_L}.  
An appropriate choice of $\eps = \{\eps_v\}$ in Lemma \ref{inequalityallplaces} gives:

\begin{prop}  \label{upper bound with L}
Fix any $L \geq 27$.  Fix $c_1$ and $c_2$ in $\Qbar$, and assume $f_{c_1}$ and $f_{c_2}$ have $N > 1$ preperiodic points in common in $\P^1(\Qbar)$.  Then for all $0 < \delta < 1/4$, we have 
	$$\<f_{c_1}, f_{c_2} \>  \leq 4 \left( \delta + \frac{3\log(1/\delta)}{2(N-1)} \right) h_L(c_1, c_2).$$
\end{prop}

\begin{proof}
Fix a number field $K$ containing $c_1$ and $c_2$.  Let $F$ be the $\Gal(\Kbar/K)$-invariant set of common preperiodic points for $f_{c_1}$ and $f_{c_2}$ in $\Qbar$, so that $|F| = N-1$.  For each place $v \in M_K$, recall the definition of $\ell_v$ from \S\ref{h_L}.  Fix $0 < \delta < 1/4$ and set 
	$$\eps_v = \delta^{3\, \ell_v}.$$
Note that $\eps_v = 1$ for all but finitely many places $v \in M_K$.

For each archimedean place $v$, note that 
	$$\eps_v = e^{-3 \, \ell_v \log(1/\delta)} = \max\{|c_1|_v, |c_2|_v, L\}^{-3 \log(1/\delta)},$$
so Proposition \ref{archimedean epsilon} implies that 
	$$\lambda_{c_i, v}(z) \leq \delta\, \ell_v$$
for any point $z$ within a neighborhood of radius $\eps_v$ of the filled Julia set $K_{c_i}$.  As all points of $F$ lie in $K_{c_i}$, this implies that 
	$$-(\mu_i, [F]_\eps)_v \leq \delta \, \ell_v$$
for this choice of $\eps_v$ and each $i$.  

Similarly for each non-archimedean place $v \nmid 2$, we apply Proposition \ref{escapebound pnot2}, and for each non-archimedean $v \mid 2$, we apply Proposition \ref{escapebound p2}.  

Summing over all places, we find that 
\begin{eqnarray*}
\sum_{v \in M_K} \frac{[K_v:\Q_v]}{[K:\Q]} \left( -(\mu_i, [F]_\eps)_v + \frac{\log(1/\eps_v)}{2|F|} \right) &\leq& \sum_v \frac{[K_v:\Q_v]}{[K:\Q]} \left(\delta \, \ell_v + \frac{3 \log(1/\delta)}{2|F|}\, \ell_v \right) \\
 &=&  \left(\delta  + \frac{3 \log(1/\delta)}{2|F|}\, \right) \; h_L(c_1, c_2).
\end{eqnarray*}
Lemma \ref{inequalityallplaces} then implies
\begin{eqnarray*}
 \< f_{c_1}, f_{c_2} \>^{1/2} &\leq& \sum_{i = 1}^2 \left( \sum_{v \in M_K} \frac{[K_v:\Q_v]}{[K:\Q]}  \left( - ( \mu_i, [F]_{\eps} )_v +\frac{\log (1/\eps_v)}{2 |F|} \right)\right)^{1/2} \\
  &\leq& 2 \left( \left(\delta  + \frac{3 \log(1/\delta)}{2|F|}\, \right) \; h_L(c_1, c_2) \right)^{1/2}.
\end{eqnarray*}
Squaring both sides yields the proposition.
\end{proof}

Now fix any $\eps$ between 0 and 1, and let $\delta = \eps/25$.  Applying Proposition \ref{upper bound with L} with $L = 27$, we have 
\begin{eqnarray*}
\<f_{c_1}, f_{c_2} \>  &\leq& 4 \left( \delta + \frac{3\log(1/\delta)}{2(N-1)} \right) h_L(c_1, c_2) \\
	&\leq& 4 \left( \delta + \frac{3\log(1/\delta)}{2(N-1)} \right) (h(c_1, c_2) + \log 16 + \log 27) \\
	&\leq& \left( \eps + \frac{C(\eps)}{N-1} \right) (h(c_1, c_2) + 1)
\end{eqnarray*}
with $C(\eps) = 40 \log(25/\eps)$.  This completes the proof of Theorem \ref{boundingN}.

\bigskip
\section{Proof of Theorem \ref{preperbound}} \label{mainproof}

In this section, we prove Theorem \ref{preperbound}, providing a uniform bound on the number of common preperiodic points for any pair $f_{c_1}$ and $f_{c_2}$ with $c_1 \not= c_2$ in $\C$.

\subsection{Proof over $\Qbar$}
Assume that $c_1$ and $c_2$ are in $\Qbar$.

We first use Theorem \ref{pairingbound} and \ref{boundingN} to provide a bound on 
	$$N := N(c_1, c_2) = |\Preper(f_{c_1})\cap\Preper(f_{c_2})|$$ 
when the height $h(c_1,c_2)$ is large.  The two theorems combined show that, if $N >1$, then it must satisfy
	$$\alpha_1 \, h(c_1, c_2) - C_1 \leq \left(\eps + \frac{C(\eps)}{N-1}\right)(h(c_1,c_2) + 1)$$
for every choice of $0 < \eps < 1$, and thus,
	$$\left( \alpha_1 - \eps - \frac{C(\eps)}{N-1}\right) (h(c_1,c_2)+1) \leq C_1 + \alpha_1.$$
Taking $\eps = \alpha_1/2$, we have
	$$\frac{\alpha_1}{2} - \frac{C(\eps)}{N-1} \leq \frac{C_1 + \alpha_1}{h(c_1,c_2) + 1}.$$
If we assume that 
	$$h(c_1, c_2) + 1 > \frac{4(C_1+\alpha_1)}{\alpha_1},$$
then the inequality becomes 
\begin{equation} \label{large height}
	N-1 < \frac{4C(\alpha_1/2)}{\alpha_1},
\end{equation}
providing a uniform bound on $N$ for all pairs $(c_1, c_2)$ of sufficiently large height.  

Now suppose that $h(c_1,c_2) +1 \leq 4(C_1+\alpha_1)/\alpha_1$.  We combine the uniform lower bound of Theorem \ref{uniformbound} with the upper bound of Theorem \ref{boundingN} to obtain
	$$\delta \leq  \left(\eps + \frac{C(\eps)}{N-1}\right)(h(c_1,c_2) + 1) \leq  \left(\eps + \frac{C(\eps)}{N-1}\right)\frac{4(C_1+\alpha_1)}{\alpha_1}$$
for any choice of $0 < \eps < 1$.  This unwinds to give 
\begin{equation} \label{bounded height}
	N-1 \leq \frac{C(\eps)}{\frac{\alpha_1\delta}{4(C_1+\alpha_1)} - \eps}.
\end{equation}
Choosing any $\eps < \alpha_1\delta/4(C_1+\alpha_1)$ gives a uniform bound on $N$.

\subsection{Proof over $\C$}  \label{proof over C}
Let $B$ denote a uniform bound on the number of common preperiodic points over all $c_1 \not= c_2$ in $\Qbar$.  Now fix $c_1$ in $\C\setminus \Qbar$.  For any $c_2 \in \C$, if $f_{c_1}$ and $f_{c_2}$ have at least one preperiodic point in common, then the field $\Q(c_1, c_2)$ must have transcendence degree 1 over $\Q$.  Moreover, if $x_1, x_2, \ldots, x_{B+1}$ denote distinct common preperiodic points for $f_{c_1}$ and $f_{c_2}$, then $k = \Q(c_1, c_2, x_1, \ldots, x_{B+1})$ will also be of transcendence degree 1, as each $x_i$ satisfies relations of the form
\begin{equation} \label{persistent}
	f_{c_1}^{n_i}(x_i) = f_{c_1}^{m_i}(x_i) \mbox{ for } n_i > m_i \geq 0 \quad\mbox{ and } \quad f_{c_2}^{k_i}(x_i) = f_{c_2}^{l_i}(x_i) \mbox{ for } k_i > l_i \geq 0.
\end{equation}
We may view $k$ as the function field $K(T)$ of an algebraic curve $T$ defined over a number field $K$.  In this way, the maps $f_{c_1}$ and $f_{c_2}$ are viewed as families of maps, parameterized by $t\in T(\C)$, and the relations \eqref{persistent} hold persistently in $t$.  

Now assume $c_2 \not = c_1$, so that the specializations $f_{c_1(t)}$ and $f_{c_2(t)}$ are distinct for all but finitely many $t \in T(\C)$.  As the elements $\{x_1, \ldots, x_{B+1}\}$ are distinct in $k$, their specializations $\{x_1(t), \ldots, x_{B+1}(t)\}$ are also distinct for all but finitely many $t$ in $T(\C)$.  In particular, this implies that we can find $t\in T(\Qbar)$ so that $c_1(t) \not= c_2(t)$ in $\Qbar$ and $f_{c_1(t)}$ and $f_{c_2(t)}$ share at least $B+1$ preperiodic points; this is a contradiction.  

Thus, the theorem is proved for all pairs $c_1 \not= c_2$ in $\C$, with the same bound as for pairs $c_1\not= c_2$ in $\Qbar$.

\section{Effective bounds on common preperiodic points} \label{effective}

In this section, we make effective Theorems  \ref{uniformbound}, \ref{pairingbound}, and \ref{boundingN}, to produce an explicit value for the bound $B$ of Theorem \ref{preperbound}:

\begin{theorem} \label{effectiveB} For all $c_1 \ne c_2 \in \mathbb{C}$, we have
$$|\Preper(f_{c_1}) \cap \Preper(f_{c_2})| \leq 10^{103}.$$
\end{theorem}

\subsection{An explicit lower bound in Theorem \ref{uniformbound}}
\label{explicit delta}

In order to provide an effective lower bound $\delta$ for Theorem \ref{uniformbound}, we need to improve our estimates on the energy pairing $E_v(c_1, c_2)$ when $|c_1 - c_2|_v$ is small at an archimedean place $v$.  Here we compute that we can take $\delta = 10^{-96}$.  

Let $H = 32001^{100/99}$.  Suppose that $c_1$ and $c_2$ lie in a number field $K$, and suppose that for at least $99/100$ of the archimedean places of $K$, we have 
	$$|c_1-c_2|_v \leq 1/H.$$
Then $h(c_1-c_2) \geq \frac{99}{100} \log H$, and the proof of Theorem \ref{weak pairingbound} implies that 
$$\< f_{c_1}, f_{c_2}\> \geq \frac{1}{16} h(c_1-c_2) - \frac{1}{16} \log(32000) \geq \frac{\log (32001/32000)}{16} > 10^{-6}.$$

Now suppose that we have $|c_1 - c_2|_v > 1/H$ for at least $1/100$ of the archimedean places of $K$.  Let $M = 9H^2$ so that 
	$$|c_1 - c_2|_v > \frac{1}{H} = \frac{3}{M^{1/2}}$$
at all of these places.  If $\max\{|c_1|_v,|c_2|_v\} > M$, then Theorem \ref{archimedeanbounds} implies that 
	$$E_v(c_1,c_2) \geq \frac{1}{64} \log M > 0.14$$
at this place $v$.  On the other hand, if $\max\{|c_1|_v,|c_2|_v\} \leq M$, we have the following bound:

\begin{prop} \label{close and bounded c}
Fix any $M\geq 1000$.   Then for all $s \geq M^2$, we have 
	$$E_\infty(c_1, c_2) \geq  \frac{|c_1-c_2|^2}{32 s^4} - \frac{117}{100}\frac{M^3}{s^6},$$
provided $\max\{|c_1|, |c_2|\} \leq M$.
\end{prop}

\noindent
Assuming Proposition \ref{close and bounded c}, we complete our computations.  With $M = 9H^2$, we have 
	$$E_v(c_1, c_2) \; \geq \; \frac{|c_1-c_2|_v^2}{32 s^4} - \frac{117 \cdot 9^3H^6}{100s^6} \; \geq \;
	\frac{1}{32 s^4 H^2}\left( 1 - \frac{117\cdot 2^5 9^3  H^8}{100 s^2} \right), $$
for all archimedean places $v$ with $|c_1 - c_2|_v > 1/H$, $\max\{|c_1|_v, |c_2|_v\} \leq 9H^2$, and $s > 9^2H^4$.  Choosing $s$ satisfying $s^2 =117\cdot 2^69^3H^8/100$, we conclude that 
	$$E_v(c_1, c_2) \; \geq \; \frac{100^2}{ 2^{18}9^6 117^2 H^{18}}$$
for all such places $v$.  This shows that, summing only over the archimedean places, we have
\begin{eqnarray*} 
\<f_{c_1}, f_{c_2} \> &\geq& \sum_{v \in M_K^\infty} \frac{[K_v:\Q_v]}{[K:\Q]} E_v(c_1, c_2) \\
 	&\geq&  \frac{1}{100} \; \min \left\{ 0.14,  \frac{100^2}{ 2^{18}9^6 117^2 H^{18}}  \right\} \; > \; 10^{-96},
\end{eqnarray*}
whenever $|c_1 - c_2|_v > 1/H$ for at least $1/100$ of the archimedean places of $K$.  This completes the computation of $\delta$, and it remains only to prove Proposition \ref{close and bounded c}.

\begin{proof}[Proof of Proposition \ref{close and bounded c}]
The result will follow from a series of elementary estimates on the values of the escape-rate functions outside the filled Julia set.  Let $\varphi_c$ be the B\"ottcher function for $f_c(z)=z^2+c$, so that $\varphi_c(f_c(z))=\varphi_c^2(z)$ for all $z$ large enough, and therefore $\varphi_c$ has expansion
\begin{equation}\label{phi expansion}\varphi_c(z)=z+\frac{c}{2z}+\cdots
\end{equation}
for $z$ near $\infty$.  We set 
   $$\lambda(z):=\lambda_{c_1}(z)-\lambda_{c_2}(z),$$
the difference of two escape-rate functions.   The energy pairing satisfies
   $$2E_\infty(c_1, c_2)=2\int_\mathbb C \lambda_{c_1}dd^c \lambda_{c_2}=-\int_\mathbb C \lambda dd^c \lambda = \int_{\mathbb C} d\lambda\wedge d^c \lambda.$$

Now fix any large $s>0$, and define $D^c_s:=\{z\in \mathbb C: |z|> s\}$.  
By Green's formula,  
  $$2E_\infty(c_1, c_2) \geq \int_{D^c_s}d\lambda \wedge d^c \lambda =-\int_{\partial D^c_s}\lambda d^c\lambda   =-\frac{1}{2\pi i}\int_{\partial D^c_s}\lambda \left(\frac{\partial \lambda}{\partial z}dz-\frac{\partial \lambda}{\partial \bar z}d\bar z\right).$$
We will estimate the latter integral.  

Note that $\lambda$ satisfies
   $$\lambda(z)=\log|\varphi_{c_1}|-\log|\varphi_{c_2}|$$
near $\infty$.  For simplicity, write $\eps:=c_1-c_2$. By the expansion \eqref{phi expansion} of $\varphi_c$
    $$2\lambda(z)=\frac{\eps}{2z^2}+\frac{\bar \eps}{2\bar z^2}+O\left(\frac{1}{|z|^3}\right).$$
Similarly, by using the Taylor expansion and letting $z=se^{i\theta}$ on the boundary $\partial D^c_s$, 
    $$2\left(\frac{\partial \lambda}{\partial z}dz-\frac{\partial \lambda}{\partial \bar z}d\bar z\right)=\left[-\left(\frac{\eps}{4s^3e^{2 i \theta}}+\frac{\bar \eps}{4s^3e^{-2 i \theta}}\right)+O\left(\frac{1}{s^4}\right)\right]isd\theta.$$
Consequently 
$$ -4\frac{1}{2\pi i}\int_{\partial D^c_s}\lambda \left(\frac{\partial \lambda }{\partial z}dz-\frac{\partial \lambda}{\partial \bar z}d\bar z\right) =  \frac{\eps\bar \eps}{4s^4}+O\left(\frac{1}{s^5}\right).$$
This gives
\begin{equation} \label{big O term}
    2E_\infty(c_1, c_2) \geq -\frac{1}{2\pi i}\int_{\partial D^c_s}\lambda \left(\frac{\partial \lambda }{\partial z}dz-\frac{\partial \lambda}{\partial \bar z}d\bar z\right) =\frac{\eps \bar \eps}{16s^4}+O\left(\frac{1}{s^5}\right)
\end{equation}
where $\eps=c_1-c_2$. To prove the proposition, we need control on the big-O term.

In the rest of this section, we fix an $M\geq 1000$. 

\begin{lemma}\label{escape estimate} Let $z, c_i\in \mathbb C$ with $|z|\geq M^2$, $|c_i|\leq M$ for $i=1,2$ and $\epsilon=c_1-c_2$.  Then
$$ \left|4 \lambda(z)-\left(\frac{\epsilon}{z^2}+\frac{\bar \epsilon}{\bar z^2}\right)\right| \leq \sum_{i=1,2}\left( \frac{202}{100}\frac{|c_i|}{|z|^4}+\frac{101}{100} \cdot \frac{|c_i|^2}{|z|^4}\right).$$
\end{lemma}

\proof First note  that for any $x\in \C$ with $|x|<1$,
\begin{equation}\label{logestimate}
 |\log(1+x)-x| = \left|-\frac{x^2}{2}+\frac{x^3}{3}+\cdots \right| \leq \frac{|x|^2}{2(1-|x|)},
\end{equation}
where the $\log(1+x)$ is taken to be the one with $-\pi/2<\mathrm{Im}(\log(1+x))<\pi/2$.
For any $|z|\geq |c|$ and $|z|> 4$, inductively it is easy to check that for each $n\geq 1$
\begin{equation}\label{fn estimate}
(|z|-|c/z|)^{2^n}\leq |f^n_c(z)|\leq (|z|+|c/z|)^{2^n},
\end{equation}
hence
    $$\log( |z|-|c/z|)\leq \lambda_c(z)\leq \log (|z|+|c/z|)$$
and 
    $$\log (|z^2+c|-|c|/|z^2+c|)\leq \lambda_c(z^2+c)=2\lambda_c(z)\leq \log (|z^2+c|+|c|/|z^2+c|).$$
Consequently for any $|z|\geq M^2$ and $|c|\leq M$, by \eqref{logestimate} one has
\begin{align*}\left|2\lambda_c(z)-\log |z^2+c|\right|&\leq \left|\log\left (1\pm \frac{|c|}{|z^2+c|^2}\right)\right|\leq   \frac{|c|}{|z^2+c|^2}+ \frac{|c|^2}{|z^2+c|^4}\frac{1}{2(1-\frac{|c|}{|z^2+c|^2})}\\&\leq \frac{101}{100}\frac{|c|}{|z|^4}.
\end{align*}
Now, by the triangle inequality and \eqref{logestimate} we have 
\begin{align*}
 \left|4\lambda(z)-\left(\frac{\epsilon}{z^2}+\frac{\bar \epsilon}{\bar z^2}\right) \right|&\leq \sum_{i=1,2}\left(\left|4\lambda_{c_i}(z)-2\log|z^2+c_i|\right|+\left|\log(z^2+c_i)-\log z^2-\frac{c_i}{z^2}\right| \right)\\
      &\, \, \, \, \,\, \, \,  +\sum_{i=1,2}\left|\log(\bar z^2+\bar c_i)-\log \bar z^2-\frac{\bar c_i}{\bar z^2}\right| \\
       &\leq \sum_{i=1,2}\left( \frac{202}{100}\frac{|c_i|}{|z|^4}+\frac{ |c_i|^2/|z|^4}{1-\left|\frac{c_i}{z^2}\right|}\right)\leq \sum_{i=1,2}\left( \frac{202}{100}\frac{|c_i|}{|z|^4}+\frac{101}{100} \cdot \frac{|c_i|^2}{|z|^4}\right).
\end{align*}
\qed

\begin{lemma}\label{ratio estimate} For any $z, c\in \mathbb C$ with $|z|\geq M^2$ and $ |c|\leq M$, we have 
$$\left| \prod_{i=1}^{n}\frac{(f^{i-1}_c(z))^2}{f^i_c(z)}-1+\frac{c}{z^2}\right|\leq  \frac{102}{100} \cdot \frac{|c|^2}{|z|^4}+ \frac{104}{100} \cdot \frac{|c|}{|z|^4}.$$
\end{lemma}
\proof For any $\alpha\in \mathbb C$ with $|\alpha|<1$, we have
  $$|e^\alpha-1|=\left|\alpha+\frac{\alpha^2}{2!}+\cdots \right|\leq |\alpha|+\frac{|\alpha|^2}{2!}+\cdots\leq \frac{|\alpha |}{1-|\alpha |}.$$
For each $i$, we always take $\log \frac{(f_c^{i-1}(z))^2}{f_c^{i}(z)}$ to be the one with  
   $$-\pi/2<\mathrm{Im}(\log \frac{(f_c^{i-1}(z))^2}{f_c^{i}(z)})<\pi/2$$
and set $\log\prod_{i=2}^{n}\frac{(f^{i-1}_c(z))^2}{f^i_c(z)}:=\sum^n_{i=2}\log\frac{(f^{i-1}_c(z))^2}{f^i_c(z)}$.
Then for each $i\geq 2$, by \eqref{logestimate} and \eqref{fn estimate}, we have
\begin{align*}
\left|\log \frac{(f_c^{i-1}(z))^2}{f_c^{i}(z)}\right|&=\left|\log \frac{1}{1+c/(f_c^{i-1}(z))^2}\right|=\left|\log \left(1+\frac{c}{(f_c^{i-1}(z))^2}\right)\right|\\
  &\leq \left|\frac{c}{(f_c^{i-1}(z))^2}\right|\left(1+\frac{\left|\frac{c}{(f_c^{i-1}(z))^2}\right|}{2\left(1-\left|\frac{c}{(f_c^{i-1}(z))^2}\right|\right)}\right)\\
  &\leq \frac{101}{100}\frac{|c|}{(|z|-|c/z|)^{2^i}},
\end{align*}
for the last inequality we use the fact that $|c|/|f_c^{i-1}(z)|^2\leq |c|/(|z|-|c/z|)^{2^i}\leq 1/1000$.
Therefore, since $(|z|-|c/z|)^2\geq M^2/2$, we conclude 
   $$\left|\log\prod_{i=2}^{n}\frac{(f^{i-1}_c(z))^2}{f^i_c(z)}\right|\leq\sum_{i=2}^n\frac{101}{100}\cdot \frac{|c|}{(|z|-|c/z|)^{2^i}}\leq \frac{102}{100}\cdot \frac{|c|}{(|z|-|c/z|)^{4}}\leq \frac{103}{100}\cdot \frac{|c|}{|z|^{4}}.$$
For $i=1$,  
   $$\left|\frac{(f_c^{i-1}(z))^2}{f_c^{i}(z)}-1+\frac{c}{z^2}\right|=\left|\frac{1}{1+\frac{c}{z^2}}-1+\frac{c}{z^2}\right|\leq \frac{|c|^2}{|z|^4}\cdot \frac{1}{1-|c/z^2|}\leq \frac{101}{100} \cdot \frac{|c|^2}{|z|^4}.$$
Finally, let 
   $$\alpha=\log\prod_{i=2}^{n}\frac{(f^{i-1}_c(z))^2}{f^i_c(z)} \textup{ and } \beta=\frac{z^2}{z^2+c}-1+\frac{c}{z^2}$$
and then 
 \begin{align*}
\left| \prod_{i=1}^{n}\frac{(f^{i-1}_c(z))^2}{f^i_c(z)}-1+\frac{c}{z^2}\right|&=\left|e^\alpha\left(\beta+1-\frac{c}{z^2}\right)-\left(1-\frac{c}{z^2}\right)\right|\\
     &\leq |e^\alpha\beta|+\left|(e^\alpha-1)\left(1-\frac{c}{z^2}\right)\right|\\
     &\leq \left(1+\frac{|\alpha|}{1-|\alpha|}\right)|\beta|+\frac{|\alpha|}{1-|\alpha|}\left(1+\frac{|c|}{|z|^2}\right).
     \end{align*}
 The inequalities for $\alpha$ and $\beta$ give  
 \begin{align*}
\left| \prod_{i=1}^{n}\frac{(f^{i-1}_c(z))^2}{f^i_c(z)}-1+\frac{c}{z^2}\right|&\leq \left(1+\frac{\frac{103}{100}\cdot \frac{|c|}{|z|^{4}}}{1-\frac{103}{100}\cdot \frac{|c|}{|z|^{4}}}\right)\cdot \frac{101}{100} \cdot \frac{|c|^2}{|z|^4}+\frac{\frac{103}{100}\cdot \frac{|c|}{|z|^{4}}}{1-\frac{103}{100}\cdot \frac{|c|}{|z|^{4}}}\left(1+\frac{|c|}{|z|^2}\right)\\
     &\leq  \frac{102}{100} \cdot \frac{|c|^2}{|z|^4}+ \frac{104}{100} \cdot \frac{|c|}{|z|^4}.
\end{align*}\qed

\begin{lemma}\label{partial estimate}
With the same hypotheses as Lemma \ref{ratio estimate}, we have that 
   $$\left|\frac{2\partial \lambda_c(z)}{\partial z}-\frac{1}{z}+\frac{c}{z^3}\right|\leq  \frac{102}{100} \cdot \frac{|c|^2}{|z|^5}+ \frac{104}{100} \cdot \frac{|c|}{|z|^5}.$$
\end{lemma}
\proof Consider   
 \begin{align*}2\lambda_c(z)&=\lim_{n\to \infty}\frac{2\log^+|f_c^n(z)|}{2^n}=\lim_{n\to \infty}\frac{\log \left(f_c^n(z) \cdot  f_{\bar c}^n(\bar z)\right) }{2^n}\\
 &=\lim_{n\to \infty}\left(\frac{\log f_c^n(z)}{2^n}+\frac{\log f_{\bar c}^n(\bar z)}{2^n}\right),
\end{align*}
and take partial derivatives of both sides, so that we have
   $$\frac{2\partial \lambda_c(z)}{\partial z} = \lim_{n\to \infty} \frac{2\partial \log^+ |f_c^n(z)|}{\partial z}=\lim_{n\to \infty}\frac{1}{2^n}\frac{\partial \log f_c^n(z)}{\partial z}=\lim_{n\to\infty}\frac{1}{2^n}\frac{\prod_{i=1}^nf'_c(f_c^{i-1}(z))}{f_c^n(z)}$$
which is independent on the choices of $\log f_c^n(z)$ and $\log f_{\bar c}^n(\bar z)$. 
Combining this with Lemma  \ref{ratio estimate}, we conclude that
\begin{align*}
\left|\frac{2\partial \lambda_c(z)}{\partial z}-\frac{1}{z}+\frac{c}{z^3}\right|&=\left|\lim_{n\to\infty}\frac{1}{2^n}\frac{\prod_{i=1}^nf'_c(f_c^{i-1}(z))}{f_c^n(z)}-\frac{1}{z}+\frac{c}{z^3}\right|\\
  &=\left|\frac{1}{z}\lim_{n\to\infty}\left(\prod_{i=1}^{n}\frac{(f^{i-1}_c(z))^2}{f^i_c(z)}-1+\frac{c}{z^2}\right)\right| \\
  &\leq  \frac{102}{100} \cdot \frac{|c|^2}{|z|^5}+ \frac{104}{100} \cdot \frac{|c|}{|z|^5}.
\end{align*}\qed

Similarly 
\begin{equation}\label{partial bar estimate}\left|\frac{2\partial \lambda_c(z)}{\partial \bar z}-\frac{1}{\bar z}+\frac{\bar c}{\bar z^3}\right|\leq  \frac{102}{100} \cdot \frac{|c|^2}{|z|^5}+ \frac{104}{100} \cdot \frac{|c|}{|z|^5}.
\end{equation}

\medskip
Now we are ready to control the big-O term in \eqref{big O term}.  Write 
   $$\lambda=\left(\frac{\eps}{4z^2}+\frac{\bar \eps}{4\bar z^2}\right)+\left[\lambda-\left(\frac{\eps}{4z^2}+\frac{\bar \eps}{4\bar z^2}\right)\right],$$
   $$\frac{\partial \lambda}{\partial z}dz=\left[\left(\frac{\partial \lambda}{\partial z}+\frac{\eps}{2z^3}\right)-\frac{\eps}{2z^3}\right]dz \quad\textup{ and }\quad \frac{\partial \lambda}{\partial \bar z}d\bar z=\left[\left(\frac{\partial \lambda}{\partial \bar z}+\frac{\bar \eps}{2\bar z^3}\right)-\frac{\bar \eps}{2\bar z^3}\right]d\bar z.$$
We set 
   $$I_1=\frac{|\eps|}{2s}\max_{i=1,2}\left[ \frac{102}{100} \cdot \frac{|c_i|^2}{|s|^5}+ \frac{104}{100} \cdot \frac{|c_i|}{|s|^5}\right],$$
    $$I_2=\frac{1}{4}\sum_{i=1,2}\left( \frac{202}{100}\frac{|c_i|}{|s|^4}+\frac{101}{100} \cdot \frac{|c_i|^2}{|s|^4}\right)\max_{i=1,2}\left (\frac{102}{100} \cdot \frac{|c_i|^2}{|s|^5}+ \frac{104}{100} \cdot \frac{|c_i|}{|s|^5}\right)\cdot s,$$
and
    $$I_3 =\frac{1}{4}\sum_{i=1,2}\left( \frac{202}{100}\frac{|c_i|}{|s|^4}+\frac{101}{100} \cdot \frac{|c_i|^2}{|s|^4}\right)\frac{|\eps|}{2s^2}.$$
Lemmas \ref{escape estimate} and \ref{partial estimate} along with inequalities \eqref{partial bar estimate} and \eqref{big O term} give 
	$$2\,E_\infty(c_1,c_2) \geq \frac{\eps\bar{\eps}}{16 s^4} - 2(I_1 + I_2 + I_3).$$

By the assumptions $M\geq1000, |c_i| \leq M$ for $i = 1,2$ and $s \geq M^2$, and since $|\eps|=|c_1-c_2|\leq 2M$, we have
    $$I_{1}\leq \frac{103}{100}\cdot \frac{M^3}{s^6}, \quad I_{2}\leq \frac{1}{1000}\cdot \frac{M^3}{s^6}, \quad\textup{ and }\quad I_{3}\leq \frac{102}{800}\cdot \frac{M^3}{s^6}.$$
Therefore, 
  $$2(I_1 + I_2 + I_3) \leq \frac{234}{100}\frac{M^3}{s^6}.$$
This completes the proof of the proposition.
\end{proof}

\subsection{Explicit bound}   \label{final explicit}
As shown in the proof of Theorem \ref{pairingbound} (in \S\ref{pairingboundproof}), we have $\alpha_1 = 1/192$ and $C_1 = 3/17$ in Theorem \ref{pairingbound}, and we may take and $C(\eps) = 40 \log(25/\eps)$ in Theorem \ref{boundingN} as shown in \S\ref{boundingNproof}.  Therefore, $C(\alpha_1/2) = 40\log(50/\alpha_1) < 367$, and whenever $c_1 \ne c_2 \in \overline{\mathbb{Q}}$ so that $f_{c_1}$ and $f_{c_2}$ have $N(c_1,c_2)>1$ common preperiodic points and $h(c_1, c_2) > 139$, we have
	$$N(c_1,c_2) < 281857 < 10^6$$
from \eqref{large height}.  For the set of parameters with $h(c_1,c_2) \leq 139$, the bound we obtain is much larger, as it depends on the small $\delta$ from Theorem \ref{uniformbound}.  We can take $\delta = 10^{-96}$, as explained in \S\ref{explicit delta}.  Taking $\eps = \alpha_1\delta/(8 (C_1 + \alpha_1))$ in \eqref{bounded height}, we find that
\begin{eqnarray*}
N(c_1,c_2) -1 &\leq& \frac{8(C_1+\alpha_1) \cdot 40 \log(25/\eps)}{\alpha_1\delta} \\
	& =&  \frac{320(C_1+\alpha_1)}{\alpha_1\delta} \; \log \frac{200(C_1+\alpha)}{\alpha_1 \delta} \\
	&\leq& \frac{320\cdot 35}{\delta} \; \log \frac{200\cdot 35}{\delta} \\
	&\leq& 96\cdot 320\cdot 35 \cdot 10^{96} \; \log (200 \cdot 35 \cdot 10),
\end{eqnarray*}
so that 
	$$N(c_1,c_2) = |\Preper(f_{c_1})\cap\Preper(f_{c_2})| < 10^{103}.$$
The same bound holds for all $c_1 \not= c_2$ in $\C$, as explained in \S\ref{proof over C}.

\bigskip
\bigskip

\begin{thebibliography}{DKY}

\bibitem[AP]{Andrews:Petsche}
J.~Andrews and C.~Petsche.
\newblock {Abelian extensionsions in dynamical Galois theory}.
\newblock {Preprint, {arXiv:2001.00659v1 [math.NT]}}.

\bibitem[BE]{Baker:Eremenko}
I.~N. Baker and A.~Er\"{e}menko.
\newblock {A problem on {J}ulia sets}.
\newblock {\em Ann. Acad. Sci. Fenn. Ser. A I Math.} {\bf 12}(1987), 229--236.

\bibitem[BD]{BD:preperiodic}
M.~Baker and L.~DeMarco.
\newblock {Preperiodic points and unlikely intersections}.
\newblock {\em Duke Math. J.} {\bf 159}(2011), 1--29.

\bibitem[BR1]{Baker:Rumely:equidistribution}
M.~Baker and R.~Rumely.
\newblock {Equidistribution of small points, rational dynamics, and potential
  theory}.
\newblock {\em Ann. Inst. Fourier (Grenoble)} {\bf 56}(2006), 625--688.

\bibitem[BR2]{BRbook}
M.~Baker and R.~Rumely.
\newblock {\em Potential theory and dynamics on the {B}erkovich projective
  line}, volume 159 of {\em Mathematical Surveys and Monographs}.
\newblock American Mathematical Society, Providence, RI, 2010.

\bibitem[Bea]{Beardon:symmetries}
A.~F. Beardon.
\newblock {Symmetries of {J}ulia sets}.
\newblock {\em Bull. London Math. Soc.} {\bf 22}(1990), 576--582.

\bibitem[Ben]{Benedetto:book}
 R. ~L. Benedetto.
 \newblock{\em Dynamics in one non-archimedean variable}, volume 198 of {\em Graduate Studies in Mathematics}.
 \newblock  American Mathematical Society, 2019.

\bibitem[BBP]{Benedetto:Briend:Perdry}
R.~Benedetto, J.-Y.~Briend, and H.~Perdry.
\newblock {Dynamique des polyn\^{o}mes quadratiques sur les corps locaux}.
 \newblock {\em J. Th\'{e}or. Nombres Bordeaux.} {\bf 19}(2007), 325--336.

\bibitem[BH]{Branner:Hubbard:1}
B.~Branner and J.~H.~Hubbard.
\newblock {The iteration of cubic polynomials. {I}. {T}he global topology of
  parameter space}.
\newblock {\em Acta Math.} {\bf 160}(1988), 143--206.

\bibitem[Br]{Brolin}
H.~Brolin.
\newblock {Invariant sets under iteration of rational functions}.
\newblock {\em Ark. Mat.} {\bf 6}(1965), 103--144.

 \bibitem[CG]{CarlesonGamelin:book}
 L.~Carleson and T.~Gamelin. 
 \newblock{\em Complex Dynamics}.
 \newblock Springer-Verlag, 1993.

\bibitem[CL1]{ChambertLoir:equidistribution}
A.~Chambert-Loir.
\newblock {Mesures et \'equidistribution sur les espaces de {B}erkovich}.
\newblock {\em J. Reine Angew. Math.} {\bf 595}(2006), 215--235.

\bibitem[CL2]{ChambertLoir:survey}
A.~Chambert-Loir.
\newblock {Heights and measures on analytic spaces. {A} survey of recent
  results, and some remarks}.
\newblock In {\em Motivic integration and its interactions with model theory
  and non-{A}rchimedean geometry. {V}olume {II}}, volume 384 of {\em London
  Math. Soc. Lecture Note Ser.}, pages 1--50. Cambridge Univ. Press, Cambridge,
  2011.
  
\bibitem[CS]{Call:Silverman}
G.~S. Call and J.~H. Silverman.
\newblock {Canonical heights on varieties with morphisms}.
\newblock {\em Compositio Math.} {\bf 89}(1993), 163--205.

\bibitem[DF1]{DF:degenerations}
L.~DeMarco and X.~Faber.
\newblock {Degenerations of complex dynamical systems}.
\newblock {\em Forum of Mathematics, Sigma} {\bf 2}(2014), 36 pages.

\bibitem[DF2]{DF:degenerations2}
L.~DeMarco and X.~Faber.
\newblock {Degenerations of complex dynamical systems {II}: analytic and
  algebraic stability}.
\newblock {\em Math. Ann.} {\bf 365}(2016), 1669--1699.
\newblock With an appendix by J.~Kiwi.

\bibitem[DKY]{DKY:UMM}
L.~DeMarco, H.~Krieger, and H.~Ye.
\newblock {Uniform Manin-Mumford for a family of genus 2 curves}.
\newblock  {\em Ann. of Math.} {\bf 191}(2020), no.3, 949-1001.

\bibitem[Fa]{Favre:degenerations}
C.~Favre.
\newblock {Degeneration of endomorphisms of the complex projective space in the
  hybrid space}.
\newblock {To appear, {\em J. Inst. Math. Jussieu}.}

\bibitem[FRL]{FRL:equidistribution}
C.~Favre and J.~Rivera-Letelier.
\newblock {\'{E}quidistribution quantitative des points de petite hauteur sur
  la droite projective}.
\newblock {\em Math. Ann.} {\bf 335}(2006), 311--361.

\bibitem[Fi]{Fili:energy}
P.~Fili.
\newblock {A metric of mutual energy and unlikely intersections for dynamical
  systems}.
\newblock {Preprint, {arXiv:1708.08403v1 [math.NT]}}.

\bibitem[KS]{Kawaguchi:Silverman:pairing}
S.~Kawaguchi and J.~H.~Silverman.
\newblock {Canonical heights and the arithmetic complexity of morphisms on
  projective space}.
\newblock {\em Pure Appl. Math. Q.} {\bf 5}(2009), 1201--1217.

\bibitem[K]{Kuhne:UMM}
L.~K\"uhne.
\newblock{Equidistribution in families of Abelian varieties and uniformity}.
\newblock {Preprint, {arXiv:2101.10272v3 [math.NT]}}.

\bibitem[LP]{Levin:Przytycki}
G.~Levin and F.~Przytycki.
\newblock {When do two rational functions have the same {J}ulia set?}
\newblock {\em Proc. Amer. Math. Soc.} {\bf 125}(1997), 2179--2190.

\bibitem[Ly]{Lyubich:entropy}
M.~Lyubich.
\newblock {Entropy properties of rational endomorphisms of the {R}iemann
  sphere}.
\newblock {\em Ergodic Theory Dynamical Systems} {\bf 3}(1983), 351--385.

\bibitem[Ma]{Mazur:curves}
B.~Mazur.
\newblock {Arithmetic on curves}.
\newblock {\em Bull. Amer. Math. Soc. (N.S.)} {\bf 14}(1986), 207--259.

\bibitem[Mi]{Milnor:dynamics}
J.~Milnor.
\newblock {\em Dynamics in One Complex Variable}, volume 160 of {\em Annals of
  Mathematics Studies}.
\newblock Princeton University Press, Princeton, NJ, {T}hird edition, 2006.

\bibitem[Pa]{Pakovich:maximal}
F.~Pakovich.
\newblock {On rational functions sharing the measure of maximal entropy}.
\newblock {Preprint, {arXiv:1910.07363v2 [math.DS]}}.

\bibitem[PST]{PST:pairing}
C.~Petsche, L.~Szpiro, and T.~J. Tucker.
\newblock {A dynamical pairing between two rational maps}.
\newblock {\em Trans. Amer. Math. Soc.} {\bf 364}(2012), 1687--1710.

\bibitem[Po]{Poonen:conjecture}
B.~Poonen.
\newblock {The classification of rational preperiodic points of quadratic
  polynomials over {${\bf Q}$}: a refined conjecture}.
\newblock {\em Math. Z.} {\bf 228}(1998), 11--29.

\bibitem[Ye]{Ye:symmetries}
H.~Ye.
\newblock {Rational functions with identical measure of maximal entropy}.
\newblock {\em Adv. Math.} {\bf 268}(2015), 373--395.

\bibitem[YZ]{Yuan:Zhang:II}
X.~Yuan and S.~Zhang.
\newblock {The arithmetic Hodge index theorem for adelic line bundles II}.
\newblock {\em Preprint, {\em arXiv:1304.3539 [math.NT]}}.

\bibitem[Zh]{Zhang:adelic}
S.~Zhang.
\newblock {Small points and adelic metrics}.
\newblock {\em J. Algebraic Geom.} {\bf 4}(1995), 281--300.

\end{thebibliography}

\def\cprime{$'$}

\end{document}